\documentclass[10pt]{amsart}
\usepackage{amsmath,amssymb,mathrsfs,color}
\usepackage[upint]{stix}

\usepackage[colorlinks,
linkcolor=red,
anchorcolor=green,
citecolor=blue,
]{hyperref}

\usepackage{paralist}

\usepackage{subfigure}
\usepackage{float}
\usepackage{times}
\usepackage{tikz}
\usepackage{cases}
\usetikzlibrary{intersections}
\usepackage{xcolor}
\usepackage{bbm}

\usetikzlibrary{shadings}
\usetikzlibrary[intersections]
\usetikzlibrary{arrows,shapes}

\makeatletter
\def\setliststart#1{\setcounter{\@listctr}{#1}%
  \addtocounter{\@listctr}{-1}}
\makeatother

\makeatletter
\@addtoreset{figure}{section}
\makeatother

\usepackage{calc}
\newtheorem{The}{Theorem}[section]
\newtheorem{Cor}[The]{Corollary}
\newtheorem{Lem}[The]{Lemma}
\newtheorem{Pro}[The]{Proposition}
\theoremstyle{definition}

\newtheorem{defn}[The]{Definition}

\theoremstyle{remark}
\newtheorem{Rem}[The]{Remark}

\numberwithin{equation}{section}

\newcommand{\R}{\mathbb{R}}

\newcommand{\N}{\mathbb{N}}

\newcommand{\ep}{\varepsilon}

\makeatletter
\def\moverlay{\mathpalette\mov@rlay}
\def\mov@rlay#1#2{\leavevmode\vtop{%
   \baselineskip\z@skip \lineskiplimit-\maxdimen
   \ialign{\hfil$\m@th#1##$\hfil\cr#2\crcr}}}
\newcommand{\charfusion}[3][\mathord]{
    #1{\ifx#1\mathop\vphantom{#2}\fi
        \mathpalette\mov@rlay{#2\cr#3}
      }
    \ifx#1\mathop\expandafter\displaylimits\fi}
\makeatother

\newcommand{\SING}{\mbox{\rm Sing}\,(u)}

\title{Local singular characteristics on $\R^2$}
\author{Piermarco Cannarsa \and Wei Cheng}
\address{Dipartimento di Matematica, Universit\`a di Roma ``Tor Vergata'', Via della Ricerca Scientifica 1, 00133 Roma, Italy}
\email{cannarsa@mat.uniroma2.it}
\address{Department of Mathematics, Nanjing University, Nanjing 210093, China}
\email{chengwei@nju.edu.cn}
\date{\today}

\begin{document}

\maketitle


\begin{abstract}
The singular set of  a viscosity solution to a Hamilton-Jacobi equation is known to propagate, from any noncritical singular point, along singular characteristics which are curves satisfying certain differential inclusions. In the literature, different notions of singular characteristics were introduced. 
However, a general uniqueness criterion for singular characteristics, not restricted to mechanical systems or problems in one space dimension, is missing at the  moment. In this paper, we prove that, for a Tonelli Hamiltonian on $\R^2$, two different notions of singular characteristics coincide up to a bi-Lipschitz reparameterization. As a significant consequence, we obtain a uniqueness result for the class of singular characteristics that was introduced by Khanin and Sobolevski in the paper [On dynamics of {L}agrangian trajectories for {H}amilton-{J}acobi
  equations.
 {\em Arch. Ration. Mech. Anal.}, 219(2):861--885, 2016]. 
\end{abstract}

\section{Introduction}

This paper is devoted to study the local propagation of singularities for viscosity solutions of the Hamilton-Jacobi equations
\begin{align}
	H(x,Du(x))=0,\quad x\in \R^n,\label{eq:intro_HJs}\tag{HJ$_s$}\\
	H(x,Du(x))=0,\quad x\in \Omega,\label{eq:intro_HJ_local}\tag{HJ$_{\rm loc}$}
\end{align}
where $H$ is a Tonelli Hamiltonian in \eqref{eq:intro_HJs} and $H$ is of class $C^1$ and strictly convex in the $p$-variable in \eqref{eq:intro_HJ_local}. In \eqref{eq:intro_HJs}, we assume that $0$ on the right-hand side is  Ma\~n\'e's critical value. The existence of global weak KAM solutions of \eqref{eq:intro_HJs} was obtained in \cite{Fathi_Maderna2007}. In \eqref{eq:intro_HJ_local}, we suppose $\Omega\subset\R^n$ is a bounded domain.

Semiconcave functions are nonsmooth functions that play an important role in the study of \eqref{eq:intro_HJs} and \eqref{eq:intro_HJ_local}. For semiconcave  viscosity solutions of Hamilton-Jacobi equations, Albano and the first author proved in \cite{Albano_Cannarsa2002} that singular arcs can be selected as generalized characteristics. Recall that a Lipschitz  arc $\mathbf{x}:[0,\tau]\to\R^n$ is called a \emph{generalized characteristic} starting from $x$ for the pair $(H,u)$ if it satisfies the following:
\begin{align}\label{intro:gc}
	\begin{cases}
		\dot{\mathbf{x}}(s)\in\mathrm{co}\,H_p\big(\mathbf{x}(s),D^+u(\mathbf{x}(s))\big)&\quad \text{a.e.}\;s\in[0,\tau],\\
		\dot{\mathbf{x}}(0)=x.&
	\end{cases}
\end{align}
If $x\in\SING$---the singular set of $u$---then \cite[Theorem~5]{Albano_Cannarsa2002}  gives a sufficient condition for the existence of a generalized characteristic propagating the singularity of $u$ locally. 

The local structure of  singular (generalized) characteristics was further investigated  by the first author and Yu in \cite{Cannarsa_Yu2009}, where  
\emph{singular characteristics}  were proved more regular near the starting point than the arcs constructed in  \cite{Albano_Cannarsa2002}. 
Such additional properties will be crucial for the analysis we develop in this paper.

For any weak KAM solution $u$ of \eqref{eq:intro_HJs}, the class of  \emph{intrinsic singular} (generalized) \emph{characteristics} was introduced in \cite{Cannarsa_Cheng3} by the  authors  of this paper, building on properties of the Lax-Oleinik semi-group of positive type. Such a method allowed to construct global singular characteristics,  which we now call {\em intrinsic}. Moreover, in  \cite{Cannarsa_Cheng_Fathi2017} and \cite{Cannarsa_Cheng_Fathi2019} such an ``intrisic approach''  turned out  to be useful for  pointing out topological properties of the cut locus of $u$, including homotopy equivalence to the complement of the Aubry set (see also \cite{CCMW2019} for applications to Dirichlet boundary value problems).

In spite of its success in capturing singular dynamics, it could be argued that the relaxation procedure in the original definition of generalized characteristics---that is, the presence of the convex hull in \eqref{intro:gc}---might cause a loss of information coming from
the Hamiltonian dynamics behind. On the other hand, such a relaxation is necessary to ensure convexity of admissible velocities for the differential inclusion in \eqref{intro:gc},
since the map $x\rightrightarrows H_p(x,D^+u(x))$ fails to be convex-valued, in general. 

The most important example where the above relaxation is unnecessary is probably given by
mechanical Hamiltonians of the form 
$$H(x,p)=\frac 12\langle A(x)p,p\rangle+V(x),$$ 
where $A(x)$ is a symmetric  positive definite $n\times n$-matrix smoothly depending on $x$ and $V(x)$ is a smooth function on $\R^n$. In this case,  a much finer theory has been developed, yielding quantitative tools for the analysis of singular characteristics (\cite{ACNS2013}, \cite{Cannarsa_Mazzola_Sinestrari2015}) and their long time behaviour (\cite{Cannarsa_Chen_Cheng2019}). This is mainly due to the fact that, for a mechanical Hamiltonian, \eqref{intro:gc} reduces to the  \emph{generalized gradient system}
 \begin{align}\label{intro:ms}
	\begin{cases}
		\dot{\mathbf{x}}(t)\in A(\mathbf{x}(t))D^+u(\mathbf{x}(t))&\quad  t>0\quad\text{a.e.}\\
		\dot{\mathbf{x}}(0)=x,&
	\end{cases}
\end{align}
the solutions of which, unique for any  initial datum, form a Lipschitz semi-flow (see, e.g., \cite{Albano_Cannarsa2002}, \cite{ACNS2013}, and \cite{Cannarsa_Cheng_Zhang2014}).
Unfortunately, such a uniqueness property, which for \eqref{intro:ms}  is a simple consequence of the  quasi-dissipativity of the set-valued map $x\rightrightarrows A(x)D^+u(x)$, breaks down for a general Hamiltonian because $x\rightrightarrows H_p(x,D^+u(x))$ is no longer quasi-dissipative (see \cite{Cannarsa_Yu2009} and \cite{Stromberg2011}).

Recent significant progress in the attempt to develop a more restrictive notion of singular characteristics is due to Khanin and Sobolevski (\cite{Khanin_Sobolevski2016}). In this paper, we will call such curves \emph{strict singular characteristic} but in the literature they are also refereed to as  \emph{broken characteristics},   see \cite{Stromberg2013,Stromberg_Ahmadzadeh2014}. We now proceed to recall their definition:
given a semiconcave solution $u$ of  \eqref{eq:intro_HJ_local},  a Lipschitz singular curve $\mathbf{x}:[0,T]\to\Omega$ is called a strict singular characteristic from $x\in\SING$ if
there exists a measurable selection $p(t)\in D^+u(\mathbf{x}(t))$ such that
\begin{equation}\label{eq:intro_sgc}
	\begin{split}
		\begin{cases}
		\dot{\mathbf{x}}(t)=H_p(\mathbf{x}(t),p(t))& a.e.\ t\in[0,T],\\
		\mathbf{x}(0)=x.&
	\end{cases}
	\end{split}
\end{equation}
As already mentioned, the existence of strict singular characteristics for time dependent Hamilton-Jacobi equations  was proved in \cite{Khanin_Sobolevski2016}, where additional regularity properties of such curves were established including  right-differentiability of $\mathbf{x}$ for every $t$, 
 right-conti\-nuity of $\dot{\mathbf{x}}$, and the fact 
that $p(\cdot):[0,T]\to\R^n$ satisfies
\begin{equation}\label{eq:KS_energy}
	H(\mathbf{x}(t),p(t))=\min_{p\in D^+u(\mathbf{x}(t))}H(\mathbf{x}(t),p)\qquad\forall t\in [0,T].
\end{equation}
In  Appendix~A, we give a proof of the existence and regularity of strict characteristics for solutions to \eqref{eq:intro_HJ_local} for the reader's convenience.

In view of the above considerations, it is quite natural to raise the following questions:
\begin{enumerate}[(Q1)]
	\item What is the relation between a strict singular characteristic, $\mathbf{x}$,  and  a singular characteristic, $\mathbf{y}$, from the same initial point?
	\item What kind of uniqueness result can be proved for singular characteristics? What about strict singular characteristics?
\end{enumerate}
In this paper, we will answer the above questions in the two-dimensional case under the following additional conditions:
\begin{enumerate}[(A)]
    \item $n=2$ and $\mathbf{y}$ is Lipschitz;
	\item the initial point $x=\mathbf{y}(0)$ of the singular characteristic $\mathbf{y}$ is not a critical point with respect the pair $(H,u)$, i.e., $0\not\in H_p(x_0,D^+u(x))$;
	\item $\mathbf{y}$ is right differentiable at $0$ and
	\begin{align*}
		\dot{\mathbf{y}}^+(0)=H(x_0,p_0),
	\end{align*}
	where $p_0=\arg\min\{H(x_0,p): p\in D^+u(x)\}$;
	\item $\lim_{t\to0^+}\operatorname*{ess\ sup}_{s\in[0,t]}|\dot{\mathbf{y}}(s)-\dot{\mathbf{y}}^+(0)|=0$.
\end{enumerate}
Notice that any strict singular characteristic $\mathbf{x}$ and the singular characteristic $\mathbf{y}$ given in \cite{Cannarsa_Yu2009} (see also Proposition \ref{pro:gc_local})  satisfy conditions (A)-(D). The intrinsic singular characteristic $\mathbf{z}$ constructed in \cite{Cannarsa_Cheng3} (see also Proposition \ref{pro:y}) satisfies just conditions (A)-(C), in general.

\smallskip
The main results of this paper can be described as follows.
\begin{enumerate}[$\bullet$]
	\item For any pair of singular curves $\mathbf{x}_1$ and $\mathbf{x}_2$ satisfying condition (A)-(D), there exists $\tau>0$ and a bi-Lipschitz homeomorphism $\phi:[0,\tau]\to[0,\phi(\tau)]$ such that, $\mathbf{x}_1(\phi(t))=\mathbf{x}_2(t)$ for all $t\in[0,\tau]$. In other words, the singular characteristic starting from a noncritical point $x$ is unique up to a bi-Lipschitz reparameterization (Theorem \ref{thm:reparametrization}).
    \item In particular, if $\mathbf{x}$ is a strict singular characteristic and $\mathbf{y}$ is a singular characteristic starting from the same noncritical initial point $x$, then there exists $\tau>0$ and a bi-Lipschitz homeomorphism $\phi:[0,\tau]\to[0,\phi(\tau)]$ such that $\mathbf{y}(\phi(t))=\mathbf{x}(t)$ for all $t\in[0,\tau]$ (Corollary \ref{Cor:reparamatrization}).
    \item We have the following \emph{uniqueness} property for strict singular characteristics: let 
    $$\mathbf{x}_j:[0,T]\to\Omega\qquad (j=1,2)$$ be strict singular characteristics from the same noncritical initial point $x$. Then there exists $\tau\in(0, T]$ such that $\mathbf{x}_1(t)=\mathbf{x}_2(t)$ for all $t\in[0,\tau]$. (Theorem \ref{the:strict})
\end{enumerate}

Finally, we remark that the results of this paper cannot be applied  to intrinsic singular characteristics because of the mentioned lack of condition (D). Extra techniques will have to be developed to cover such an important class of singular arcs. 

The paper is organized as follows. In section 2, we introduce necessary material on Hamilton-Jacobi equations, semiconcavity, and singular characteristics. In section 3, we answer question (Q1)-(Q2) in the two-dimensional case. In the appendix, we give a  full proof of the existence of strict singular characteristics. 

\medskip

\noindent\textbf{Acknowledgements.} Piermarco Cannarsa was supported in part by the National Group for Mathematical Analysis, Probability and Applications (GNAMPA) of the Italian Istituto Nazionale di Alta Matematica ``Francesco Severi'' and by Excellence Department Project awarded to the Department of Mathematics, University of Rome Tor Vergata, CUP E83C18000100006. Wei Cheng is partly supported by National Natural Science Foundation of China (Grant No. 11871267, 11631006 and 11790272). The authors also appreciate the cloud meeting software Zoom for the help to finish this paper in this difficult time of coronavirus.

\section{Hamilton-Jacobi equation and semiconcavity}

In this section, we review some basic facts on semiconcave functions and Hamilton-Jacobi equations.

\subsection{Semiconcave function}

Let $\Omega\subset\R^n$ be a convex open set. We recall that  a function $u:\Omega\to\R$ is {\em semiconcave} (with linear modulus) if there exists a constant $C>0$ such that
\begin{equation}\label{eq:SCC}
\lambda u(x)+(1-\lambda)u(y)-u(\lambda x+(1-\lambda)y)\leqslant\frac C2\lambda(1-\lambda)|x-y|^2
\end{equation}
for any $x,y\in\Omega$ and $\lambda\in[0,1]$. 

Let $u:\Omega\subset\R^n\to\R$ be a continuous function. For any $x\in\Omega$, the closed convex sets
\begin{align*}
D^-u(x)&=\left\{p\in\R^n:\liminf_{y\to x}\frac{u(y)-u(x)-\langle p,y-x\rangle}{|y-x|}\geqslant 0\right\},\\
D^+u(x)&=\left\{p\in\R^n:\limsup_{y\to x}\frac{u(y)-u(x)-\langle p,y-x\rangle}{|y-x|}\leqslant 0\right\}.
\end{align*}
are called the {\em subdifferential} and {\em superdifferential} of $u$ at $x$, respectively.

The following characterization of semiconcavity (with linear modulus) for a continuous function comes from proximal analysis. 

\begin{Pro}
\label{criterion-Du_semiconcave}
Let $u:\Omega\to\R$ be a continuous function. If there exists a constant $C>0$ such that, for any $x\in\Omega$, there exists $p\in\R^n$ such that
\begin{equation}\label{criterion_for_lin_semiconcave}
u(y)\leqslant u(x)+\langle p,y-x\rangle+\frac C2|y-x|^2,\quad \forall y\in\Omega,
\end{equation}
then $u$ is semiconcave with constant $C$ and $p\in D^+u(x)$.
Conversely,
if $u$ is semiconcave  in $\Omega$ with constant $C$, then \eqref{criterion_for_lin_semiconcave} holds for any $x\in\Omega$ and $p\in D^+u(x)$.
\end{Pro}

Let $u:\Omega\to\R$ be locally Lipschitz. We recall that a vector $p\in\R^n$ is called a {\em reachable} (or {\em limiting}) {\em gradient}  of $u$ at $x$ if there exists a sequence $\{x_n\}\subset\Omega\setminus\{x\}$ such that $u$ is differentiable at $x_k$ for each $k\in\N$, and
$$
\lim_{k\to\infty}x_k=x\quad\text{and}\quad \lim_{k\to\infty}Du(x_k)=p.
$$
The set of all reachable gradients of $u$ at $x$ is denoted by $D^{\ast}u(x)$.

The following proposition concerns fundamental properties of semiconcave funtions and their gradients (see \cite{Cannarsa_Sinestrari_book} for the proof).
\begin{Pro}\label{basic_facts_of_superdifferential}
Let $u:\Omega\subset\R^n\to\R$ be a semiconcave function and let $x\in\Omega$. Then the following properties hold.
\begin{enumerate}[\rm {(}a{)}]
  \item $D^+u(x)$ is a nonempty compact convex set in $\R^n$ and $D^{\ast}u(x)\subset\partial D^+u(x)$, where  $\partial D^+u(x)$ denotes the topological boundary of $D^+u(x)$.
  \item The set-valued function $x\rightsquigarrow D^+u(x)$ is upper semicontinuous.
  \item If $D^+u(x)$ is a singleton, then $u$ is differentiable at $x$. Moreover, if $D^+u(x)$ is a singleton for every point in $\Omega$, then $u\in C^1(\Omega)$.
  \item $D^+u(x)=\mathrm{co}\, D^{\ast}u(x)$.
  \item If $u$ is both semiconcave and semiconvex in $\Omega$, then $u\in C^{1,1}(\Omega)$.
\end{enumerate}
\end{Pro}

\begin{defn}
Let $u:\Omega\to\R$ be a semiconcave function. $x\in\Omega$ is called a \emph{singular point} of $u$ if $D^+u(x)$ is not a singleton. The set of all singular points of $u$ is denoted by $\SING$.
\end{defn}

\begin{defn}
Let $k\in\{0,1,\dots,n\}$ and let $C\subset\R^n$. $C$ is called a \emph{$k$-rectifiable} set if there exists a Lipschitz continuous function $f:\R^k\to\R^n$ such that $C\subset f(\R^k)$.  $C$ is called a \emph{countably $k$-rectifiable} set if it is the union of a countable family of $k$-rectifiable sets.
\end{defn}

Let us recall a result on the rectifiability of the singular set $\SING$ of a semiconcave function $u$ in dimension two.

\begin{Pro}[\cite{Cannarsa_Sinestrari_book}]\label{pro:rectifiability}
Let $\Omega\subset\R^2$ be an open domain, $u:\Omega\to\R$ be a semiconcave function, and set 
\begin{align*}
	\mbox{\rm Sing}_k(u)=\{x\in\SING: \dim(D^+u(x))=k\},\quad k=0,1,2.
\end{align*}
Then $\mbox{\rm Sing}_k(u)$ is countably $(2-k)$-rectifiable for $k=0,1,2$. In particular, $\mbox{\rm Sing}_2(u)$ is countable.
\end{Pro}

\subsection{Aspects of weak KAM theory}

For any $x,y\in\R^n$ and $t>0$, we denote by $\Gamma^t_{x,y}$ the set of all absolutely continuous curves $\xi$ defined on $[0,t]$ such that $\xi(0)=x$ and $\xi(t)=y$. Define
\begin{equation}\label{eq:def_f_s}
	A_t(x,y)=\inf_{\xi\in\Gamma^t_{x,y}}\int^t_0L(\xi(s),\dot{\xi}(s))\ ds,\quad x,y\in\R^n,\ t>0.
\end{equation}
We call $A_t(x,y)$ the \emph{fundamental solution} for the Hamilton-Jacobi equation
\begin{align*}
	D_tu(t,x)+H(x_0,D_xu(t,x))=0,\quad t>0, x\in\R^n.
\end{align*}
By  classical results (Tonelli's theory),  the infimum in \eqref{eq:def_f_s} is a minimum. Each curve $\xi\in\Gamma^t_{x,y}$ attaining such a minimum is called a \emph{minimal curve for $A_t(x,y)$}.

\begin{defn}
For each $u:\R^n\to\R$, let and $T_tu$ and $\breve{T}_tu$ be the \emph{Lax-Oleinik evolution of negative and positive type} defined, respectively, by 
\begin{align*}
	\begin{split}
		T_tu(x)=&\,\inf_{y\in\R^n}\{u(y)+A_t(y,x)\},\\
		\breve{T}_tu(x)=&\,\sup_{y\in\R^n}\{u(y)-A_t(x,y)\},
	\end{split}
	\qquad (x\in\R^n, t>0).
\end{align*}
\end{defn}
The following result is well-known.
\begin{Pro}[\cite{Fathi_Maderna2007}]\label{pro:weak_KAM}
There exists a  Lipschitz  semiconcave viscosity solution of \eqref{eq:intro_HJs}. Moreover, such a solution $u$ is a common fixed point of the semigroup $\{T_t\}$, i.e., $T_tu=u$ for all $t\geqslant0$.
\end{Pro}
Clearly, \eqref{eq:intro_HJs} has no unique solution and we call each solution, given as a fixed point of the semigroup $\{T_t\}$, a \emph{weak KAM solution} of \eqref{eq:intro_HJs}.

\begin{defn}
Let $u$ be a continuous function on $M$. We say $u$ is \emph{$L$-dominated} if 
\begin{align*}
	u(\xi(b))-u(\xi(a))\leqslant \int^b_aL(\xi(s),\dot{\xi}(s))\ ds,
\end{align*}
for all absolutely continuous curves $\xi:[a,b]\to\R^n\;(a<b)$, with $\xi(a)=x$ and $\xi(b)=y$. We say such an absolutely continuous curve $\xi$ is a \emph{$(u,L)$-calibrated curve}, or a \emph{$u$-calibrated curve} for short, if the equality holds in the inequality above. A curve $\xi:(-\infty,0]\to\R^n$ is called a $u$-calibrated curve if it is $u$-calibrated  on each compact sub-interval of $(-\infty,0]$. In this case, we also say that $\xi$ is a \emph{backward calibrated curve} (with respect to $u$).
\end{defn}

The following result explains the relation between  the set of all reachable gradients and the set of all backward calibrated curves from $x$ (see, e.g., \cite{Cannarsa_Sinestrari_book} or \cite{Rifford2008} for the proof). 

\begin{Pro}\label{reachable_grad_and_backward}
Let $u:\R^n\to\R$ be a weak KAM solution of  \eqref{eq:intro_HJs} and let  $x\in \R^n$. Then $p\in D^{\ast}u(x)$ if and only if there exists a unique $C^2$ curve $\xi:(-\infty,0]\to \R^n$ with $\xi(0)=x$ and $p=L_v(x,\dot{\xi}(0))$, which is a backward calibrated curve with respect to $u$.
\end{Pro}

\subsection{Propagation of singularities}

In this paper, we will discuss various types of singular arcs  describing the propagation of  singularities for Lipschitz semiconcave solutions of the Hamilton-Jacobi equations \eqref{eq:intro_HJ_local} and \eqref{eq:intro_HJs}. 

\begin{defn}
$x$ is called a \emph{critical point with respect to $(H,u)$} if $0\in H_p(x_0,D^+u(x))$.
\end{defn}
Let $u$ be a Lipschitz  semiconcave viscosity solution of \eqref{eq:intro_HJ_local} and $x\in\SING$. 
\begin{defn}\label{def:sc}
A \emph{singular characteristic from $x$}  is   a Lipschitz arc $\mathbf{x}:[0,\tau]\to\Omega\;(\tau>0))$ such that:
\begin{enumerate}[(1)]
	\item $\mathbf{x}$ is a generalized characteristic with $\mathbf{x}(0)=x$,
	\item $\mathbf{x}(t)\in\SING$ for all $t\in[0,\tau]$,
	\item $\dot{\mathbf{x}}^+(0)=H_p(x_0,p_0)$ where $p_0=\arg\min\{H(x_0,p): p\in D^+u(x)\}$,
	\item $\lim_{t\to0^+}\operatorname*{ess\ sup}_{s\in[0,t]}|\dot{\mathbf{x}}(s)-\dot{\mathbf{x}}^+(0)|=0$.
\end{enumerate}
\end{defn}

The following existence of singular characteristic is due to \cite{Albano_Cannarsa2002,Cannarsa_Yu2009}.

\begin{Pro}\label{pro:gc_local}
Let $u$ be a Lipschitz  semiconcave  solution of \eqref{eq:intro_HJ_local} and $x\in\SING$. Then, there exists a singular characteristic $\mathbf{y}:[0,T]\to\Omega$ with $\mathbf{y}(0)=x$.
\end{Pro}

Now, suppose $u$ is a Lipschitz  semiconcave weak KAM solution of \eqref{eq:intro_HJs}. In \cite{Cannarsa_Cheng3}, another singular curve for $u$  is constructed as follows. First, it is shown that there exists $\lambda_0>0$ such that for any $(t,x)\in \R_+\times \R^n$ and any maximizer $y$  for the function $u(\cdot)-A_t(x,\cdot)$, we have that $|y-x|\leqslant\lambda_0 t$. Then, taking $\lambda=\lambda_0+1$, one shows  that there exists $t_{0}>0$ such that, if $t\in(0,t_0]$, then there exists a unique  $y_{t,x}\in B(x,\lambda t)$ of $u(\cdot)-A_t(x,\cdot)$ such that
\begin{equation}\label{eq:sup_max_rep}
	\breve{T}_tu(x)=u(y_{t,x})-A_t(x,y_{t,x}).
\end{equation}
Moreover, such a $t_0$ is  such that $-A_t(x,\cdot)$ is concave with constant $C_2/t$ and $C_1-C_2/t<0$ for $0<t\leqslant t_0$. We now define the curve
\begin{equation}\label{eq:curve_max}
	\mathbf{z}(t)=
	\begin{cases}
		x,&t=0,\\
		y_{t,x},& t\in(0,t_0].
	\end{cases}
\end{equation}

\begin{Pro}[\cite{Cannarsa_Cheng3}]\label{pro:y}
Let  the curve $\mathbf{z}$ be defined in \eqref{eq:curve_max}. Then, the following  holds:
\begin{enumerate}[\rm (1)]
	\item $\mathbf{z}$ is a Lipschitz generalized characteristic,
	\item if $x\in\SING$ then $\mathbf{z}(t)\in\SING$ for all $t\in[0,t_0]$,
	\item $\dot{\mathbf{z}}^+(0)$ exists and
	\begin{align*}
		\dot{\mathbf{z}}^+(0)=H_p(x_0,p_0)
	\end{align*}
	where $p_0=\arg\min\{H(x_0,p): p\in D^+u(x)\}$.
\end{enumerate}  
\end{Pro}
\noindent
Hereafter, we will refer to the arc $\mathbf{z}$ defined in \eqref{eq:curve_max} as the {\em intrinsic characteristic} from $x$.

\section{Singular characteristic on $\R^2$}

We now return to questions (Q1) and (Q2) from the Introduction. So far, we have introduced three kinds of singular arcs issuing from a point $x_0\in\SING$, namely
\begin{itemize}[$\bullet$]
\item strict singular characteristics, that is, solutions  to \eqref{eq:intro_sgc},
\item singular characteristics, introduced in Definition~\ref{def:sc}, and
\item the intrinsic singular characteristic $\mathbf{z}$ given  by Proposition \ref{pro:y}.
\end{itemize}
In this section, we will compare the first two notions of characteristics when $\Omega\subset\R^2$. 

We begin by introducing the following class of Lipschitz arcs.
\begin{defn}
 Given $T>0$, we denote by $\mbox{\rm Lip}_0(0,T;\Omega)$ the class of all Lipschitz arcs ${\mathbf{x}}:[0,T]\to\Omega$ such that
  the right derivative 
  $$\dot{\mathbf{x}}^+(0)=\lim_{t\downarrow 0}\frac{\mathbf{x}(t)-\mathbf{x}(0)}t$$ 
	does exist and satisfies
\begin{equation}
\label{right-continuity}
\lim_{t\to0^+}\operatorname*{ess\ sup}_{s\in[0,t]}|\dot{\mathbf{x}}(s)-\dot{\mathbf{x}}^+(0)|=0.
\end{equation}
\end{defn}

For any  $\mathbf{x}\in\mbox{\rm Lip}_0(0,T;\Omega)$  we set
\begin{equation}
\label{omega}
\omega_{\mathbf{x}}(t):=\operatorname*{ess\ sup}_{s\in[0,t]}|\dot{\mathbf{x}}(s)-\dot{\mathbf{x}}^+(0)|.
\end{equation}
Owing to \eqref{right-continuity}, we have that $\omega_{\mathbf{x}}(t)\to 0$ as $t\downarrow 0$.

\begin{Lem}\label{le:injectivity}
Let $\mathbf{x}\in\mbox{\rm Lip}_0(0,T;\Omega)$  be such that $\dot{\mathbf{x}}^+(0)\neq 0$. Then, 
\begin{equation}\label{eq:X_phi}
	\big| |\mathbf{x}(t_1)-\mathbf{x}(t_0)|-|t_1-t_0|\cdot|\dot{\mathbf{x}}^+(0)|\big| 
	\leqslant|t_1-t_0|\omega_{\mathbf{x}}(t_1\vee   t_0)\quad\forall t_0,t_1\in [0,T]
\end{equation} 
and $\mathbf{x}$ is injective on some interval $[0,T_0]$ with $0<T_0<T$.
\end{Lem}

\begin{proof}
Observe that, for any $0\leqslant t_0\leqslant t_1\leqslant T$, the identity
\begin{equation*}
	\mathbf{x}(t_1)-\mathbf{x}(t_0)=\int^{t_1}_{t_0}\dot{\mathbf{x}}(t)\ dt=(t_1-t_0)\dot{\mathbf{x}}^+(0)+\int^{t_1}_{t_0}(\dot{\mathbf{x}}(t)-\dot{\mathbf{x}}^+(0))\ dt
\end{equation*}
immediately gives \eqref{eq:X_phi}. In turn, \eqref{eq:X_phi} implies that, if $\mathbf{x}(t_1)-\mathbf{x}(t_0)=0$, then
\begin{align*}
	|t_1-t_0|\cdot|\dot{\mathbf{x}}^+(0)|\leqslant|t_1-t_0|\omega_{\mathbf{x}}(t_1)
\end{align*}
Since $\dot{\mathbf{x}}^+(0)\not=0$, we conclude that $t_1=t_0$ if $ t_0, t_1\in[0, T_0]$ with $T_0$ sufficiently small.  
\end{proof}

Let $x\in\R^2$ and let $\theta\in\R^n$ be a unit vector. For any $\rho \in(0,1)$ let us consider the cone
\begin{equation}
\label{eq:cone}
C_\rho (x,\theta)=\big\{y\in\R^2~\big|~|\langle y-x,\theta\rangle|\geqslant \rho |y-x| \big\}
\end{equation}
with vertex in $x$, amplitude $\rho$, and axis $\theta$. Clearly, $C_\rho (x,\theta)$ is given by the union of the two cones
\begin{equation*}
C^+_\rho (x,\theta)=\big\{y\in\R^2~\big|~\langle y-x,\theta\rangle\geqslant \rho |y-x| \big\}
\end{equation*}
and
\begin{equation*}
C^-_\rho (x,\theta)=\big\{y\in\R^2~\big|~\langle y-x,\theta\rangle\leqslant -\rho |y-x| \big\},
\end{equation*}
which intersect each other only at $x$.
\begin{Lem}\label{lem:cone}
Let $\mathbf{x}_j\in\mbox{\rm Lip}_0(0,T;\Omega)$ ($j=1,2$) be such that
\begin{enumerate}[\rm (i)]
	\item  $\mathbf{x}_1(0)= \mathbf{x}_2(0)=:x_0$, 
	\item $\dot{\mathbf{x}}_1^+(0)=\dot{\mathbf{x}}_2^+(0)$, and
	\item  $\dot{\mathbf{x}}_j(s)\neq 0$  ($j=1,2$) for a.e. $s\in[0,T]$.
	\end{enumerate}
Define
\begin{equation}
\label{eq;theta1}
\theta_1(s)= \frac{\dot{\mathbf{x}}_1(t)}{|\dot{\mathbf{x}}_1(t)|}\qquad(s\in[0,T]\mbox{ a.e.})
\end{equation}
and fix $\rho\in (0,1)$. Then the following holds true:
\begin{enumerate}[\rm (a)]
	\item there exists $s_\rho  \in(0,T]$ such that $x_0\in C^-_\rho (\mathbf{x}_1(s),\theta_1(s))$ for a.e. $s\in[0,s_\rho  ]$;
	\item there exists $\tau_\rho \in(0,T]$ such that for all $t\in(0,\tau_\rho]$ there exists $\sigma_\rho  (t)\in (0,T]$ such that
	\begin{gather}
		|\mathbf{x}_2(t)-\mathbf{x}_1(s)|\leqslant  \frac{1+\rho}{2\rho} t|\dot{\mathbf{x}}_1^+(0)|\quad \forall s\in[0,\sigma_\rho  (t)]\label{eq:b1}\\
		\mathbf{x}_2(t)\in C^+_\rho (\mathbf{x}_1(s),\theta_1(s))\quad\mbox{for a.e. }s\in[0,\sigma_\rho  (t)].\label{eq:b2}
	\end{gather}
\end{enumerate}
\end{Lem}

\begin{proof}
Hereafter, we denote by  $o_i(s)\,(i\in\N)$ any (scalar- or vector-valued) function such that  
$$\lim_{s\to0^+}\frac{o_i(s)}s=0.$$
In view of \eqref{eq:X_phi} we conclude that 
\begin{equation}
\label{eq:base}
|x_0-\mathbf{x}_1(s)|=s|\dot{\mathbf{x}}_1^+(0)|+o_1(s)\quad\forall s\in[0,T].
\end{equation}
Moreover, setting $\theta_1(0)=\dot{\mathbf{x}}_1^+(0)/|\dot{\mathbf{x}}_1^+(0)|$, for a.e. $s\in [0,T]$ we have that
\begin{equation}\label{eq:base2}
	\begin{split}
		\langle x_0-\mathbf{x}_1(s),\theta_1(s)\rangle=&\,\langle x_0-\mathbf{x}_1(s),\theta_1(0)\rangle+\langle x_0-\mathbf{x}_1(s),\theta_1(s)-\theta_1(0)\rangle\\
		=&\,-s|\dot{\mathbf{x}}^+_1(0)|+o_2(s).
	\end{split}
\end{equation}
Now, having fixed $\rho \in(0,1)$ let  $s_\rho  \in(0,T_1]$ be such that, for a.e. $s\in [0,s_\rho  ]$,
\begin{equation*}
\frac{|o_1(s)|}s \leqslant \frac{1-\rho }{2\rho} |\dot{\mathbf{x}}^+_1(0)|\quad\mbox{and}\quad\frac{|o_2(s)|}s \leqslant \frac{1-\rho }2 |\dot{\mathbf{x}}^+_1(0)|.
\end{equation*}
Then $|x_0-\mathbf{x}_1(s)|\leqslant \frac{1+\rho}{2\rho}s|\dot{\mathbf{x}}_1^+(0)|$ by \eqref{eq:base}. From  \eqref{eq:base2} it follows that
\begin{equation*}
\langle x_0-\mathbf{x}_1(s),\theta_1(s)\rangle
\leqslant -\frac{1+\rho }2 s |\dot{\mathbf{x}}^+_1(0)| \leqslant -\rho  |x_0-\mathbf{x}_1(s)|
	\qquad(s\in [0,s_\rho  ]\mbox{ a.e.})
\end{equation*}
and (a) follows.

The proof of (b) is similar: since  $\dot{\mathbf{x}}_2^+(0)=\dot{\mathbf{x}}_1^+(0)$ by condition (ii), for all  $t\in[0,T]$ and $s\in[0,T]$ we have  that 
\begin{equation}
\label{eq:inizio}
\mathbf{x}_2(t)-\mathbf{x}_1(s)=(t-s)\dot{\mathbf{x}}_1^+(0) +o_3(t)+o_3(s).
\end{equation}
Hence, for  all  $s,t\in(0,T]$  we deduce that 
\begin{equation*}
\Big|\frac{\mathbf{x}_2(t)-\mathbf{x}_1(s)}{t|\dot{\mathbf{x}}_1^+(0)|}- \frac{\dot{\mathbf{x}}^+_1(0)}{|\dot{\mathbf{x}}^+_1(0)|}\Big|\leqslant\,
\frac st+\frac{|o_3(t)|+|o_3(s)|}{t|\dot{\mathbf{x}}_1^+(0) |}.
\end{equation*}
So,
\begin{equation}
\label{eq:equiv_bis}
 |\mathbf{x}_2(t)-\mathbf{x}_1(s)|\leqslant t|\dot{\mathbf{x}}_1^+(0)|
 \Big(
 1+ \frac st+\frac{|o_3(t)|+|o_3(s)|}{t|\dot{\mathbf{x}}_1^+(0) |}\Big).
\end{equation}

Next,  take the scalar product of each side of \eqref{eq:inizio} with $\theta_1(s)$ to obtain
\begin{multline}
\label{s_t}
	\langle\mathbf{x}_2(t)-\mathbf{x}_1(s),\theta_1(s)\rangle
	=t\langle \dot{\mathbf{x}}_1^+(0),\theta_1(s)\rangle-\langle s\dot{\mathbf{x}}_1^+(0)-o_3(t)-o_3(s),\theta_1(s)\rangle
\\
=t|\dot{\mathbf{x}}_1^+(0)|+t \langle \dot{\mathbf{x}}_1^+(0),\theta_1(s)-\theta_1(0)\rangle
-\langle s\dot{\mathbf{x}}_1^+(0)-o_3(t)-o_3(s),\theta_1(s)\rangle
\end{multline}
for all  $t\in[0,T]$ and a.e. $s\in[0,T]$. 

Once again, having fixed $\rho \in(0,1)$,
we can find $\tau_\rho \in(0,T]$ satisfying the following: for all $t\in(0,\tau_\rho]$ there exists $\sigma_\rho (t)\in (0,T]$ such that
\begin{equation*}
t| \langle \dot{\mathbf{x}}_1^+(0),\theta_1(s)-\theta_1(0)\rangle|+
|\langle s\dot{\mathbf{x}}_1^+(0)-o_3(t)-o_3(s),\theta_1(s)\rangle|
\leqslant \frac{1-\rho }2 t|\dot{\mathbf{x}}_1^+(0)|
\end{equation*}
and
\begin{equation*}
1+ \frac st+\frac{|o_3(t)|+|o_3(s)|}{t|\dot{\mathbf{x}}_1^+(0) |}\leqslant \frac{1+\rho}{2\rho}
\end{equation*}
for all  $t\in[0,\tau_\rho ]$ and a.e. $s\in[0,\sigma_\rho (t)]$. Then, \eqref{eq:equiv_bis} leads directly to \eqref{eq:b1}.
 Moreover, returning to \eqref{s_t}, for all  $t\in[0,\tau_\rho ]$ and a.e. $s\in[0,\sigma_\rho (t)]$ we conclude that
\begin{equation*}
\langle\mathbf{x}_2(t)-\mathbf{x}_1(s),\theta_1(s)\rangle
\geqslant t|\dot{\mathbf{x}}_1^+(0)|-\frac{1-\rho }2 t|\dot{\mathbf{x}}_1^+(0)|
	=\frac{1+\rho }2 t|\dot{\mathbf{x}}_1^+(0)|
	\geqslant\rho|\mathbf{x}_2(t)-\mathbf{x}_1(s)|,
\end{equation*}
where we have used \eqref{eq:equiv_bis} to deduce the last inequality. Hence, \eqref{eq:b2} follows. 
\end{proof}

Given  a  semiconcave  solution $u$ of \eqref{eq:intro_HJ_local}, we hereafter concentrate on singular arcs for $u$, that is, arcs $\mathbf{x}\in\mbox{\rm Lip}_0(0,T;\Omega)$ such that $\mathbf{x}(t)\in\SING$ for all $t\in[0,T]$. We denote such a subset of $\mbox{\rm Lip}_0(0,T;\Omega)$ by $\mbox{\rm Lip}^u_0(0,T;\Omega)$.

\begin{Lem}\label{lem:superdiff}
Let $u$ be a  semiconcave  solution of \eqref{eq:intro_HJ_local}  and let $\mathbf{x}\in\mbox{\rm Lip}^u_0(0,T;\Omega)$  be such that $\dot{\mathbf{x}}^+(0)\neq 0$. Then there exists $T_0\in (0,T]$ such that  the set 
$$S_{\mathbf{x}}=\Big\{s\in[0,T_0]~\big|~D^+u(\mathbf{x}(s))=[p_s^1,p_s^2]\;\mbox{with}\; p_s^1, p_s^2\in D^*u(\mathbf{x}(s))\,,\;p_s^1\neq p_s^2\Big\}.$$ 
has full measure in $[0,T_0]$. Moreover, 	 $\lim_{s\to0^+}p_s^i=p_0^i$ with  $p_0^i\in D^*u(x_0)\;(i=1,2)$ and 
\begin{align}\label{perp}
	\langle \dot{\mathbf{x}}(s),p_s^2-p_s^1\rangle=0\quad\mbox{for a.e. } s\in [0,T_0]
\end{align}
\end{Lem}

\begin{proof}
 The structure of the superdifferential of $u$ along $\mathbf{x}$ is described by Proposition \ref{pro:rectifiability} and Proposition 3.3.15 in \cite{Cannarsa_Sinestrari_book}. 
\end{proof}
 
 \begin{Lem}\label{lem:calibrated}
Let $u$ be a  semiconcave  solution of \eqref{eq:intro_HJ_local} and let  $x_0\in \SING$ be such that $0\not\in H_p(x_0,D^+u(x_0))$. Let $\mathbf{x}\in\mbox{\rm Lip}_0^u(0,T;\Omega)$  be such that $\mathbf{x}(0)=x_0$ and
  \begin{align*}
		\dot{\mathbf{x}}^+(0)=H_p(x_0,p_0)\quad\text{where}\ p_0=\arg\min\{H(x_0,p): p\in D^+u(x_0)\}.
	\end{align*}
Let $T_0\in (0,T]$ be given by Lemma~\ref{lem:superdiff} and, for every $s\in S_\mathbf{x}$, let $\xi^1_s$ and  $\xi^2_s$ 
be  backward calibrated curves 
on $(-\infty,0]$  satisfying
\begin{equation}
\label{eq:calib1}
\xi_s^i(0)=\mathbf{x}(s)\quad\mbox{ and }\quad 
\dot{\xi}_s^i(0)=H_p(\mathbf{x}(s), p^i_{s})  \qquad(i=1,2)
\end{equation}
Then there exist constants  $r_1>0$, $s_1\in(0,T_0]$, and  $\delta\in(0,1)$ and such that  
\begin{equation}
\label{eq:calib2}
|\mathbf{x}(s)-\xi^i_s(-r)|\geqslant \delta r \qquad(i=1,2)
\end{equation}
and, for all $s\in [0,s_1]\cap S_\mathbf{x}$ and $r\in [0,r_1]$,
\begin{equation}
\label{eq:calib3}
\xi^1_s(-r)\in C^+_\delta (\mathbf{x}(s),\theta_2(s))\quad\mbox{and}\quad  \xi^2_s(-r)\in C^-_\delta (\mathbf{x}(s),\theta_2(s))
\end{equation}
where 
\begin{equation*}
\theta_2(s)=\frac{p^2_s-p^1_s}{|p^2_s-p^1_s|}\qquad(s\in S_\mathbf{x}).
\end{equation*}
\end{Lem}

\begin{proof}
 The existence of backward calibrated curves satisfying \eqref{eq:calib1} follows from Proposition~\ref{reachable_grad_and_backward}.
 Moreover, for all $r\geqslant 0$ we have that 
\begin{equation}
\label{epistassi}
\mathbf{x}(s)-\xi^i_s(-r)=\xi^i_s(0)-\xi^i_s(-r)=r\dot{\xi}^i(0)+o(r)= rH_p(\mathbf{x}(s), p^i_{s})+o(r) \quad(i=1,2)
\end{equation}
where  $\lim_{r\to0^+}o(r)/r=0$ uniformly with respect to $s\in S_\mathbf{x}$. 

Now, observe that, since $x_0$ is not a critical point with respect to $(u,H)$, by possibly reducing $T_0$ we have that $\mathbf{x}(s)$ is also not a critical point for all $s\in[0,T_0]$ due to the upper-semicontinuity of the set-valued map $s\rightrightarrows H_p(\mathbf{x}(s),D^+u(\mathbf{x}(s)))$.
So,  for some $r_0>0$,  $s_0\in(0,T_1]$, and $\delta_0\in (0,1)$,
we deduce that
\begin{equation}
\label{eq:twocones}
\frac r{\delta_0}\geqslant |\mathbf{x}(s)-\xi^i_s(-r)|=r|H_p(\mathbf{x}(s), p^i_{s})|+o(r)\geqslant \delta_0 r\quad(i=1,2)
\end{equation}
for all $s\in [0,s_0]\cap S_\mathbf{x}$ and $r\in [0,r_0]$. This proves \eqref{eq:calib3}.

Next, recall that $H(x_0,p^i_0)=0$ because $p_0^i\in D^*u(x_0)\;(i=1,2)$. So,
by the strict convexity of $H(x_0,\cdot)$, we deduce that there exists $\nu>0$ such that
\begin{align*}
	\langle H_p(x_0,p^2_0),p^2_0-p^1_0\rangle\geqslant&\,H(x_0,p^2_0)-H(x_0,p^1_0)+\nu|p^2_0-p^1_0|^2=\nu|p^2_0-p^1_0|^2>0\\
	\langle H_p(x_0,p^1_0),p^2_0-p^1_0\rangle\leqslant&\,H(x_0,p^2_0)-H(x_0,p^1_0)-\nu|p^2_0-p^1_0|^2=-\nu|p^2_0-p^1_0|^2<0
\end{align*}
Hence, the upper-semicontinuity of the set-valued map $s\rightrightarrows H_p(\mathbf{x}(s),D^+u(\mathbf{x}(s)))$ ensures the existence of  numbers $\delta_1\in(0,1)$  and $s_1\in (0,s_0]$ such that
\begin{equation}\label{eq:2cones}
	\langle H_p(\mathbf{x}(s),p^2_s),\theta_2(s)\rangle\geqslant\delta_1,\quad\langle H_p(\mathbf{x}(s),p^1_s),\theta_2(s)\rangle\leqslant-\delta_1
	\quad\forall s\in[0,s_1]\cap S.
\end{equation}
Therefore, combining \eqref{epistassi} and  \eqref{eq:2cones}, we conclude that, after possibly replacing $r_0$ by a smaller nummber $r_1>0$,
\begin{equation*}
\langle \xi_s(-r)-\mathbf{x}(s),\theta_2(s)\rangle=-r\langle H_p(\mathbf{x}(s),p^1_s),\theta_2(s)\rangle+o(r)
\geqslant r\delta_1+o(r)\geqslant r\frac{\delta_1}2
\end{equation*}
for all $s\in [0,s_1]\cap S_\mathbf{x}$ and $r\in [0,r_1]$. By \eqref{eq:twocones}  and the above inequality we have that  $\xi^1_s(-r)\in C^+_\delta (\mathbf{x}(s),\theta_2(s))$
with $\delta= \delta_0\delta_1/2$. 

The analogous statement  for $\xi^2_s$ in \eqref{eq:calib3} can be proved  by a similar argument.
\end{proof}

We are now ready to state our main result, which ensures that singular curves coincide up to a bi-Lipschitz reparameterization, at least when $x$ is not a critical point.

\begin{The}\label{thm:reparametrization}
Let $u$ be a  semiconcave  solution of \eqref{eq:intro_HJ_local} and let  $x_0\in \SING$ be such that $0\not\in H_p(x_0,D^+u(x_0))$. Let $\mathbf{x}_j\in\mbox{\rm Lip}_0^u(0,T;\Omega)$  ($j=1,2$)  be such that $\mathbf{x}_j(0)=x_0$ and
  \begin{align*}
		\dot{\mathbf{x}}_j^+(0)=H_p(x_0,p_0)\mbox{ where } p_0=\arg\min\{H(x_0,p): p\in D^+u(x_0)\}.
	\end{align*}
Then, there exists $\sigma\in(0,T]$ such that there exists a unique  bi-Lipschitz homeomorphism $$\phi:[0,\sigma]\to[0,\phi(\sigma)]\subset [0,T_2]$$ satisfying  $\mathbf{x}_1(s)=\mathbf{x}_2(\phi(s))$ for all $s\in[0,\sigma]$. 
\end{The}

We begin the proof with the following lemma.

\begin{Lem}\label{lem:cross}
Under all assumptions of Theorem~\ref{thm:reparametrization}, there exists $\sigma\in(0,T]$ such that for all $s\in[0,\sigma]$ there exists  a unique $t_s\in[0,T]$ satisfying $\mathbf{x}_2(t_s)=\mathbf{x}_1(s)$.
\end{Lem}

\begin{proof}
First,   reduce $T>0$ in order to ensure that  $\mathbf{x}_1$ and $\mathbf{x}_2$ are both injective on $[0,T]$ and satisfy  $\dot{\mathbf{x}}_j(s)\neq 0$  for a.e. $s\in[0,T]$ ($j=1,2$).

Then,  observe that  Lemma~\ref{lem:calibrated}, applied to $\mathbf{x}=\mathbf{x}_1$, ensures the existence of  $r_1>0$, $s_1\in(0,T]$, and  $\delta\in(0,1)$ such that for a.e. $s\in [0,s_1]$ 
one can find backward calibrated curves 
 $\xi^1_s$ and  $\xi^2_s$  on $(-\infty,0]$  satisfying \eqref{eq:calib1}, \eqref{eq:calib2}, and \eqref{eq:calib3} for all  $r\in [0,r_1]$.

Next, choose 
$$
\rho=\frac{1+\sqrt{1-\delta^2}}2\in \big( \sqrt{1-\delta^2},1\big)
$$
in Lemma~\ref{lem:cone} and let $s_\rho$, $\tau_\rho$, and $\sigma_\rho(\cdot)$ be such that
\begin{enumerate}[(i)]
	\item $x_0\in C^-_\rho (\mathbf{x}_1(s),\theta_1(s))$ for a.e. $s\in[0,s_\rho  ]$,
	\item $\mathbf{x}_2(t)\in C^+_\rho (\mathbf{x}_1(s),\theta_1(s))$ for all $t\in [0,\tau_r]$ and a.e. $s\in[0,\sigma_\rho  (t)]$,
	\item $|\mathbf{x}_2(t)-\mathbf{x}_1(s)|\leqslant  \frac{1+\rho}{2\rho} t|\dot{\mathbf{x}}_1^+(0)|$ for all $t\in [0,\tau_r]$ and all  $s\in[0,\sigma_\rho  (t)]$.
\end{enumerate}

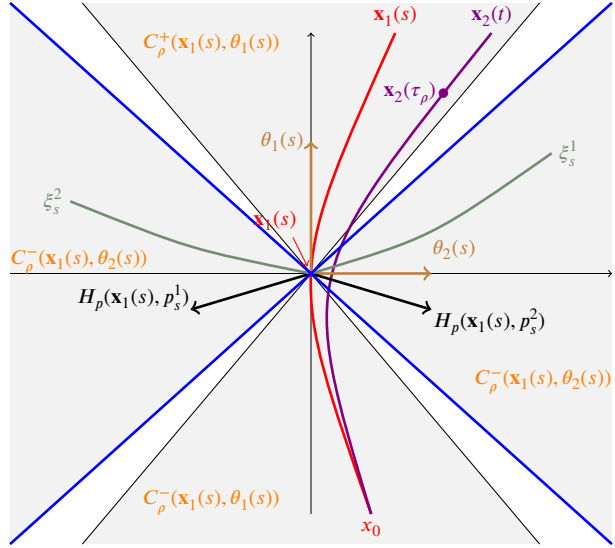
\begin{figure}\label{fig}
	\begin{center}
\begin{tikzpicture}[scale=1.6]
    \filldraw[fill=gray!10!white]
    (-2.5,2.25) -- (0,0) -- (-2.5,-2.25);
    \filldraw[fill=gray!10!white]
    (2.5,2.25) -- (0,0) -- (2.5,-2.25);
    \filldraw[fill=gray!10!white]
    (-1.9,2.25) -- (0,0) -- (1.9,2.25);
    \filldraw[fill=gray!10!white]
    (-1.9,-2.25) -- (0,0) -- (1.9,-2.25);
    \draw [line width=.1pt] [->,black][name path=axis_x] (-2.5,0) .. controls (0,0) .. (2.5,0);
	\draw [line width=.1pt] [->,black][name path=axis_y] (0,-2) .. controls (0,0) .. (0,2);
	\draw [line width=1pt] [red][name path=curve x] (0.5,-2) .. controls (-.2,0) .. (0.7,2);
	\draw [line width=1pt] [violet][name path=curve y] (0.5,-2) .. controls (-.1,0) .. (1.5,2);
	\draw [line width=1pt] [->] [brown][name path=curve v] (0,0) .. controls (1,0) .. (1,0);
	\draw [line width=1pt] [->] [black][name path=curve H_p1] (0,0) .. controls (0.5,-0.15) .. (1,-0.3);
	\draw [line width=1pt] [->] [black][name path=curve H_p2] (0,0) .. controls (-0.5,-0.15) .. (-1,-0.3);
	\draw [line width=1pt] [->] [brown][name path=curve v2] (0,0) .. controls (0,1) .. (0,1.1);
	\draw [line width=1pt] [green!10!gray][name path=curve xi1] (0,0) .. controls (1,0.3) .. (2,1);
	\draw [line width=1pt] [green!10!gray][name path=curve xi2] (0,0) .. controls (-1,0.2) .. (-2,0.6);
	\draw [line width=1pt] [blue][name path=curve 1] (0,0) .. controls (-1,0.9) .. (-2.5,2.25);
	\draw [line width=1pt] [blue][name path=curve 3] (0,0) .. controls (-1,-0.9) .. (-2.5,-2.25);
	\draw [line width=1pt] [blue][name path=curve 2] (0,0) .. controls (1,0.9) .. (2.5,2.25);
	\draw [line width=1pt] [blue][name path=curve 4] (0,0) .. controls (1,-0.9) .. (2.5,-2.25);
	\filldraw [violet] (1.1,1.5) circle [radius=1pt];
	\draw (0.95,-0.4) node[right][color=black][scale=0.8] {$H_p(\mathbf{x}_1(s),p^2_{s})$};
	\draw (-0.95,-0.2) node[left][color=black][scale=0.8] {$H_p(\mathbf{x}_1(s),p^1_{s})$};
	\draw (0.95,0.2) node[right][color=brown][scale=0.8] {$\theta_2(s)$};
	\draw (0.5,-2) node[below][color=red][scale=0.8] {$x_0$};
	\draw (2,1) node[right][color=green!10!gray][scale=0.8] {$\xi^1_{s}$};
	\draw (-2,0.6) node[left][color=green!10!gray][scale=0.8] {$\xi^2_{s}$};
	\draw (-0.3,0.3) node[above][color=red][scale=0.8] {$\mathbf{x}_1(s)$};
	\draw[->] [red] (-0.15,0.3) -- (-0.03,0.06);
	\draw (1.5,2) node[above][color=violet][scale=0.8] {$\mathbf{x}_2(t)$};
	\draw (0.7,2) node[above][color=red][scale=0.8] {$\mathbf{x}_1(s)$};
	\draw (0,1.1) node[left][color=brown][scale=0.8] {$\theta_1(s)$};
	\draw (1.3,-0.9) node[right][color=orange][scale=0.8] {$C^-_\rho (\mathbf{x}_1(s),\theta_2(s))$};
	\draw (-1.3,0.1) node[left][color=orange][scale=0.8] {$C^-_\rho (\mathbf{x}_1(s),\theta_2(s))$};
	\draw (-0.2,-1.9) node[left][color=orange][scale=0.8] {$C^-_\rho (\mathbf{x}_1(s),\theta_1(s))$};
	\draw (-0.2,1.9) node[left][color=orange][scale=0.8] {$C^+_\rho (\mathbf{x}_1(s),\theta_1(s))$};
	\draw (1.1,1.5) node[left][color=violet][scale=0.8] {$\mathbf{x}_2(\tau_{\rho})$};
\end{tikzpicture}
\end{center}
\caption{The illustration of various objects near $\mathbf{x}_1(s)$ for sufficiently small $s>0$.}\label{fig}
\end{figure}
By possibly reducing $\tau_\rho$, without loss of generality we can suppose that 
\begin{equation}
\label{eq:ball}
 \frac{1+\rho}{2\rho} \tau_\rho|\dot{\mathbf{x}}_1^+(0)|<\delta r_1.
\end{equation}
Then, recalling that 
\begin{equation*}
\theta_1(s)= \frac{\dot{\mathbf{x}}_1(s)}{|\dot{\mathbf{x}}_1(s)|}
\quad\mbox{and}\quad
\theta_2(s)=\frac{p^2_s-p^1_s}{|p^2_s-p^1_s|}\qquad(s\in[0,T]\mbox{ a.e.})
\end{equation*}
are orthogonal unit vectors, we claim  that, for a.e. $0\leqslant s\leqslant s_1\wedge \sigma_\rho(\tau_\rho)$,
\begin{equation*}
C_\rho (\mathbf{x}_1(s),\theta_1(s))\bigcap C_\delta (\mathbf{x}_1(s),\theta_2(s))=\{\mathbf{x}_1(s)\}.
\end{equation*}
Indeed, for any $x\in C_\rho (\mathbf{x}_1(s),\theta_1(s))\cap C_\delta (\mathbf{x}_1(s),\theta_2(s))$  we have that
\begin{align*}
	|x-\mathbf{x}_1(s)|^2=&\,\langle x-\mathbf{x}_1(s),\theta_1(s)\rangle^2+\langle x-\mathbf{x}_1(s),\theta_2(s)\rangle^2\\
	\geqslant&\,(\rho^2+\delta^2)|x-\mathbf{x}_1(s)|^2.
\end{align*}
This yields  $x=\mathbf{x}_1(s)$ because  $\rho^2+\delta^2>1$.

Now, define
$
\sigma=\min\big\{s_1, s_\rho, \sigma_\rho(\tau_\rho)\big\}
$
and fix $ s\in [0,\sigma]$ in the set of full measure on which (i) is satisfied together with
(ii) and (iii), that is, 
\begin{equation*}
\mathbf{x}_2(\tau_\rho)\in C^+_\rho (\mathbf{x}_1( s),\theta_1( s))
\quad\mbox{and}\quad
|\mathbf{x}_2(\tau_\rho)-\mathbf{x}_1( s)|<\delta r_1
\end{equation*}
where  \eqref{eq:ball} has also been taken into account. By possibly reducing $\sigma$, we also have that
$|\mathbf{x}_2(t)-\mathbf{x}_1( s)|<\delta r_1
$ for  all $t\in[0,\tau_\rho]$.
So, the arc $\mathbf{x}_2$, restricted to $[0,\tau_\rho]$, 
connects the point $\mathbf{x}_2(\tau_\rho)$ of the cone $C^+_\rho (\mathbf{x}_1( s),\theta_1( s))$ with 
$x_0\in C^-_\rho (\mathbf{x}_1( s),\theta_1( s))$, remaining in the open ball of radius $\delta r_1$ centered at $\mathbf{x}_1( s)$. Thus, in view of  \eqref{eq:calib2} and \eqref{eq:calib3}, $\mathbf{x}_2$ must intersect at least one of the two calibrated curves  $\xi^1_s$ and  $\xi^2_s$. However, this can happen only at $\xi^1_s(0)=\mathbf{x}_1( s)=\xi^2_s(0)$,  because $u$ is smooth at all points $\xi^2_s(-r)$ with $0<r<\infty$, whereas $\mathbf{x}_2$ is a singular arc. Finally, such an intersection occurs at a unique time $t_s$ owing to Lemma~\ref{le:injectivity}.

To complete the proof we observe that  $\mathbf{x}_2(t_s)=\mathbf{x}_1(s)$ for all $s\in[0,\sigma]$, not just on a set of full measure. This fact can be easily justified by an approximation argument.
\end{proof}

We are now in a position to prove our main result.

\smallskip

\begin{proof}[Proof of Theorem~\ref{thm:reparametrization}]
Let $\sigma\in(0,T]$ be given by Lemma \ref{lem:cross}. Then for each $s\in[0,\sigma]$ there exists a unique $\phi(s):=t_s\in[0,T_1]$ with $\mathbf{x}_2(\phi(s))=\mathbf{x}_1(s)$. 

Recalling that, thanks to Lemma~\ref{le:injectivity}, both $\mathbf{x}_1(\cdot)$ and $\mathbf{x}_2(\cdot)$ can be assumed to be injective on $[0,\sigma]$ and $[0,\phi(\sigma)]$, respectively, we proceed to show that $\phi$ is also an injection. Observe that, for any $0\leqslant s_0, s_1\leqslant\sigma$,
\begin{align*}
	\mathbf{x}_2(\phi(s_1))-\mathbf{x}_2(\phi(s_0))=&\,\int^{\phi(s_1)}_{\phi(s_0)}\dot{\mathbf{x}}_2(t)\ dt\\
	=&\,\int^{\phi(s_1)}_{\phi(s_0)}(\dot{\mathbf{x}}_2(t)-\dot{\mathbf{x}}_2^+(0))\ dt+(\phi(s_1)-\phi(s_0))\dot{\mathbf{x}}_2^+(0).
\end{align*}
Therefore,
\begin{align*}
	|\mathbf{x}_2(\phi(s_1))-\mathbf{x}_2(\phi(s_0))-(\phi(s_1)-\phi(s_0))\dot{\mathbf{x}}_2^+(0)|\leqslant\omega_{\mathbf{x}_2}(\phi(s_1)\vee\phi(s_0))|\phi(s_1)-\phi(s_0)|,
\end{align*}
where $\omega_{\mathbf{x}_2}$ is given by \eqref{omega}. Thus, returning to $\mathbf{x}_1=\mathbf{x}_2\circ\phi$ we derive
\begin{equation}\label{eq:Y_phi}
	\begin{split}
		|\mathbf{x}_1(s_1)-\mathbf{x}_1(s_0)|\geqslant&\,|\phi(s_1)-\phi(s_0)|\big(|\dot{\mathbf{x}}_2^+(0)|-\omega_{\mathbf{x}_2}(\phi(s_1)\vee\phi(s_0))|\big),\\
	|\mathbf{x}_1(s_1)-\mathbf{x}_1(s_0)|\leqslant&\,|\phi(s_1)-\phi(s_0)|\big(|\dot{\mathbf{x}}_2^+(0)|+\omega_{\mathbf{x}_2}(\phi(s_1)\vee\phi(s_0))|\big).
	\end{split}
\end{equation}
Notice that \eqref{eq:Y_phi} leads to
\begin{equation}
\label{eq:bilip}
	|\phi(s_1)-\phi(s_0)|\geqslant\frac{|\mathbf{x}_1(s_1)-\mathbf{x}_1(s_0)|}{|\dot{\mathbf{x}}_2^+(0)|+\omega_{\mathbf{x}_2}(\phi(s_1)\vee\phi(s_0))|}
\end{equation}
and this implies that $\phi$ is injective as so is $\mathbf{x}_1$.

Next, we prove that $\phi$ is continuous on $[0,\sigma]$, or the graph of $\phi$ is closed. Let $s_j\to\bar{s}$ be any sequence  such that $\phi(s_j)\to\bar{t}$ as $j\to\infty$. 
Then  
\begin{equation*}
\mathbf{x}_1(s_j)\to\mathbf{x}_1(\bar{s})
\quad\mbox{and}\quad
\mathbf{x}_2(\phi(s_j))=\mathbf{x}_1(s_j)\to\mathbf{x}_2(\bar{t})
\quad\mbox{as}\quad
j\to\infty.
\end{equation*}
So, $\mathbf{x}_2(\phi(\bar{s}))=\mathbf{x}_1(\bar{s})=\mathbf{x}_2(\bar{t})$. Since $\mathbf{x}_2(\cdot)$ is  injective, it follows that $\bar{t}=\phi(\bar{s})$. 

Being continuous,  $\phi$ is  a homeomorphism. It remains to prove that $\phi$ is bi-Lipschitz.  The  continuity of $\phi$ at $0$  ensures that, after possibly reducing $\sigma$,
\begin{equation}
\label{eq:omega_small}
\omega_{\mathbf{x}_1}(s),\,\omega_{\mathbf{x}_2}(t)\leqslant\frac{|\dot{\mathbf{x}}_2^+(0)|}2=\frac{|\dot{\mathbf{x}}_1^+(0)|}2
\end{equation}
for all $s_0,s_1\in[0,\sigma]$. Thus, by \eqref{eq:Y_phi} we have that
\begin{align*}
	|\phi(s_1)-\phi(s_0)|\leqslant\frac{|\mathbf{x}_1(s_1)-\mathbf{x}_1(s_0)|}{|\dot{\mathbf{x}}_2^+(0)|-\omega_{\mathbf{x}_2}(\phi(s_1)\vee\phi(s_0))|}
	\leqslant\frac{2\,\mbox{\rm Lip}\,(\mathbf{x}_1)}{|\dot{\mathbf{x}}_1^+(0)|}\cdot|s_1-s_0|
	\end{align*}
for all $s\in[0,\sigma]$ and $t\in[0,\phi(\sigma)]$. So, $\phi$ is Lipschitz on $[0,\sigma]$. The fact that  $\phi^{-1}$ is also Lipschitz follows by a similar argument. Indeed, writing \eqref{eq:bilip} for $t_i=\phi(s_i)$ and appealing to Lemma~\ref{le:injectivity} and \eqref{eq:omega_small} once again we obtain
\begin{eqnarray*}
|t_1-t_0|&\geqslant&\frac{|\mathbf{x}_2(t_1)-\mathbf{x}_2(t_0)|}{|\dot{\mathbf{x}}_2^+(0)|+\omega_{\mathbf{x}_2}(t_1\vee t_0)}=\frac{|\mathbf{x}_1(s_1)-\mathbf{x}_1(s_0)|}{|\dot{\mathbf{x}}_2^+(0)|+\omega_{\mathbf{x}_2}(t_1\vee t_0)}
\\
&\geqslant&\frac{|\dot{\mathbf{x}}_1^+(0)|-\omega_{\mathbf{x}_1}(s_1\vee s_0)}{|\dot{\mathbf{x}}_2^+(0)|+\omega_{\mathbf{x}_2}(t_1\vee t_0)}\cdot|s_1-s_0|
\geqslant \frac 13 \cdot|s_1-s_0|
\end{eqnarray*}

The proof is completed noting that $\phi$ is unique due to the injectivity of $\mathbf{x}_1$ and $\mathbf{x}_2$. 
\end{proof}

\begin{Cor}\label{Cor:reparamatrization}
Let $\mathbf{x}$ be a strict singular characteristic as in \eqref{eq:intro_sgc} and let $\mathbf{y}$ be any singular characteristic as in Proposition \ref{pro:gc_local}. If $x_0$ is not a critical point with respect to $(H,u)$, then there exists $\sigma>0$ and a bi-Lipschitz homeomorphism $\phi:[0,\sigma]\to[0,\phi(\sigma)]$ such that $\mathbf{y}(\phi(s))=\mathbf{x}(s)$ for all $s\in[0,\sigma]$.
\end{Cor}

For strict singular characteristics, uniqueness holds without reparameterization as we show next.

\begin{The}\label{the:strict}
Let $u$ be a  semiconcave  solution of \eqref{eq:intro_HJ_local} and let  $x_0\in \SING$ be such that $0\not\in H_p(x_0,D^+u(x_0))$. Let $\mathbf{x}_j:[0,T]\to\Omega$ ($j=1,2$) be  strict singular characteristics with initial point $x_0$. Then there exists $\tau\in(0, T]$ such that $\mathbf{x}_1(t)=\mathbf{x}_2(t)$ for all $t\in[0,\tau]$.
\end{The}
\begin{proof}
By  Theorem~\ref{thm:reparametrization}  there exists a   bi-Lipschitz homeomorphism $\phi:[0,\tau_1]\to[0,\tau_2]$, with $0\leqslant \tau_j\leqslant T\,(j=1,2)$, such that 
\begin{equation}
\label{eq:uniqhe}
\mathbf{x}_1(t)=\mathbf{x}_2(\phi(t))\qquad\forall t\in[0,\tau_1].
\end{equation}
Moreover, since $\mathbf{x}_1$ and $\mathbf{x}_2$ are strict characteristics we have that
\begin{equation*}
\begin{cases}
\dot{\mathbf{x}}^+_j(t)=H_p(\mathbf{x}_j(t),p_j(t))
 \\
 H(\mathbf{x}_j(t),p_j(t))=\min_{p\in D^+u(\mathbf{x}_j(t))} H(\mathbf{x}_j(t),p)
\end{cases}
\quad\forall t\in [0,\tau_j]\, (j=1,2)
\end{equation*}
Therefore,
\begin{equation*}
H_p(\mathbf{x}_1(t),p_1(t))=\phi'(t)H_p(\mathbf{x}_2(\phi(t)),p_2(\phi(t)))\qquad( t\in [0,\tau_1])
\end{equation*}
where, in addition to \eqref{eq:uniqhe}, we have that 
\begin{equation*}
p_2(\phi(t)))=\mbox{arg}\min_{p\in D^+u(\mathbf{x}_2(\phi(t))} H(\mathbf{x}_2(\phi(t)),p)
=\mbox{arg}\min_{p\in D^+u(\mathbf{x}_1(t)} H(\mathbf{x}_1(t),p)=p_1(t).
\end{equation*}
So, $H_p(\mathbf{x}_1(t),p_1(t))=\phi'(t)H_p(\mathbf{x}_1(t),p_1(t))$ for all $ t\in [0,\tau_1]$.
Since $0\not\in H_p(x_0,D^+u(x_0))$, we conclude that $\phi'(t)=1$, or
$\phi(t)=t$, on some interval $0\leqslant t\leqslant \tau\leqslant\tau$.
\end{proof}
Theorem~\ref{thm:reparametrization} and Theorem~\ref{the:strict} establish a connection between the absence of critical points and uniqueness of strict singular characteristics. In this direction, we also have the following global result.
\begin{Cor}\label{cor:strict}
Let $u$ be a  semiconcave  solution of \eqref{eq:intro_HJ_local} and let  $x_0\in \SING$. Let $\mathbf{x}_j:[0,T]\to\Omega$ ($j=1,2$) be  strict singular characteristics with initial point $x_0$ such that $0\not\in H_p(\mathbf{x}_j(t),D^+u(\mathbf{x}_j(t)))$ for all $t\in[0,T]$. Then $\mathbf{x}_1(t)=\mathbf{x}_2(t)$ for all $t\in[0,T]$.
\end{Cor}
\begin{proof}
On account of Theorem~\ref{the:strict} we have that 
\begin{equation*}
\mathcal T:=\big\{\tau\in(0,T]~|~\mathbf{x}_1(t)=\mathbf{x}_2(t)\,,\;\forall t\in[0,\tau]\big\}
\end{equation*}
is a nonempty set. Let $\tau_0=\sup\mathcal T=\max \mathcal T$. We claim that $\tau_0=T$. For if $\tau_0<T$, applying Theorem~\ref{the:strict} with initial point $\mathbf{x}_1(\tau_0)$ we  conclude that $\mathbf{x}_1(t)=\mathbf{x}_2(t)$  on some intarval $\tau_0\leqslant t<\tau_0+\delta$, contradicting the definition of $\tau_0$.
\end{proof}
Another well-known example where we have uniqueness of the generalized characteristic is the mechanical Hamiltonian
\begin{equation}\label{eq:mech_Hamiltonian}
	H(x,p)=\frac 12|p^2|+V(x),\quad (x,p)\in\Omega\times\R^n,
\end{equation}
with $V$ a smooth function on $\Omega$.  More precisely, if $x\in\SING$, then there exists a unique Lipschitz arc $\mathbf{y}$ determined by $\dot{\mathbf{y}}^+(t)=p(t)$, where $\mathbf{y}(0)=x$ and  $p(t)=\arg\min_{p\in D^+u(\mathbf{y}(t))}|p|$. 
 In this case, uniqueness follows from semiconcavity by an application of Gronwall's lemma (see, e.g., \cite{Cannarsa_Sinestrari_book}) ensuring that, in addition, any generalised characteristic is strict. 
We now give another  justification of such a property from the point of view of  this section. 
\begin{Cor}
If $H$ is a mechanical Hamiltonian as in \eqref{eq:mech_Hamiltonian}, then the reparameterization $\phi$ in Theorem \ref{thm:reparametrization} is the  identity.
\end{Cor}

\begin{proof}
We observe that, for almost all $t\geqslant0$,
\begin{align*}
	\dot{\mathbf{y}}(t)=\lambda(t)p_0(t)+(1-\lambda(t))p_1(t)
\end{align*}
where $\lambda(t)\in[0,1]$ and we can assume $D^+u(\mathbf{y}(t))$ is a segment, say $[p_1(t),p_0(t)]$, or $\{p_0(t),p_1(t)\}\in D^*u(\mathbf{y}(t))$. Notice that $\{p_0(t),p_1(t)\}$ is also the set of extremal points of the convex set $D^+u(\mathbf{y}(t))$.

Since $\mathbf{x}(t)=\mathbf{y}(\phi(t))$, differentiating we obtain that
\begin{align*}
	\dot{\mathbf{x}}(t)=p(t)=&\,\phi'(t)\dot{\mathbf{y}}(\phi(t))\\
	=&\,\phi'(t)\{\lambda(\phi(t))p_0(\phi(t))+(1-\lambda(\phi(t)))p_1(\phi(t))\}
\end{align*}
with $D^+u(\mathbf{y}(\phi(t)))=[p_0(\phi(t)),p_1(\phi(t))]$, or $\{p_0(\phi(t)),p_1(\phi(t))\}\in D^*u(\mathbf{y}(\phi(t)))$.

Therefore, there exists a unique $\lambda_t\in[0,t]$ such that
\begin{align*}
	p(t)=\lambda_tp_0(\phi(t))+(1-\lambda_t)p_1(\phi(t)).
\end{align*}
It follows that
\begin{align*}
	\phi'(t)=\phi'(t)\{\lambda(\phi(t))+(1-\lambda(\phi(t))))=\lambda_t+(1-\lambda_t)=1.
\end{align*}
Thus, $\phi(t)\equiv t$ and this completes the proof.
\end{proof}

\appendix

\section{Existence of strict singular characteristics}
In this Appendix, we prove the following result which ensures the existence of strict singular characteristics mentioned in the Introduction.

We recall that 
$$\dot{\mathbf{x}}^+(t):=\lim_{h\downarrow 0}\frac{\mathbf{x}(t+h)-\mathbf{x}(t)}h\qquad (t\in [0,T))$$
denotes the right derivative of $\mathbf{x}:[0,T]\to\Omega$, whenever such a derivative exists.
\begin{The}\label{app_main}
Let $u$ be a  semiconcave  solution of \eqref{eq:intro_HJ_local}. If  $x_0\in \SING$ satisfies 
\begin{equation}
\label{hp:noncritical}
0\not\in \mbox{\em  co }H_p(x_0,D^+u(x_0)),
\end{equation}
then there exists a Lipschitz singular arc $\mathbf{x}:[0,T]\to\Omega$ and a right-continuous selection $p(t)\in D^+u(\mathbf{x}(t))$ such that
\begin{equation}\label{eq:app_sgc}
	\begin{split}
		\begin{cases}
		\dot{\mathbf{x}}^+(t)=H_p(\mathbf{x}(t),p(t))& \forall t\in[0,T),\\
		\mathbf{x}(0)=x_0.&
	\end{cases}
	\end{split}
\end{equation}
and
\begin{equation}\label{app:KS_energy}
	H(\mathbf{x}(t),p(t))=\min_{p\in D^+u(\mathbf{x}(t))}H(\mathbf{x}(t),p)\qquad\forall t\in [0,T).
\end{equation}
\end{The}
\begin{Rem}
 The existence of strict singular characteristics for time dependent  Hamilton-Jacobi equations was proved by Khanin and Sobolevski under the additional assumption that the solution $u$ can be locally represented as the minimum of a compact family of smooth functions. Theorem~\ref{app_main} adapts \cite[Theorem~2]{Khanin_Sobolevski2016} to stationary equations removing such an extra assumption. 
\end{Rem}

\begin{proof} The proof, which uses ideas from \cite{Khanin_Sobolevski2016}, requires several intermediate steps.

\smallskip
Let $R_0>0$ be such that the closed ball $B(x_0,2R_0)$ is contained in $\Omega$.  Take any  sequence of smooth functions  $u_m:B(x_0,2R_0)\to\R$  such that 
\begin{equation*}
\begin{cases}
(a)  & u_m\stackrel{m\to\infty}{\longrightarrow} u
\;\;
 \mbox{uniformly on }B(x_0,R_0)
 \vspace{.1cm}
 \\
(b)
 &
  \max\{\|Du\|_{\infty}, \|Du_m\|_{\infty}\}\leqslant C_1
  \\
  (c)&
 D^2u_m\leqslant C_2I
\end{cases}
\end{equation*}
for some constants $C_1,C_2>0$. A sequence with the above properties can be constructed in several ways, for instance by using mollifiers like in  \cite{Yu2006,Cannarsa_Yu2009}.
In view of the above uniform bounds, there exists $T_0>0$ such that for any $m\geqslant 1$ the Cauchy problem
\begin{equation}\label{eq:approx_GC}
	\begin{cases}
		\dot{\mathbf{x}}(t)=H_p(\mathbf{x}(t),Du_m(\mathbf{x}(t))),&t\in[0,T_0]\\
		\mathbf{x}(0)=x_0
	\end{cases}
\end{equation}
has a unique solution $\mathbf{x}_m:[0,T_0]\to B(x_0,R_0)$. Moreover, by possibly taking a subsequence, we can assume that $\mathbf{x}_m$ converges uniformly on $[0,T_0]$ to some Lipschitz arc $\mathbf{x}:[0,T_0]\to B(x_0,R_0)$. We will show that, after possibly replacing  $T_0$ by a smaller $T>0$, such a limiting curve $\mathbf{x}$ has the required properties. 
\begin{Lem}
\label{lma:noncritical}
For every $\bar t\in [0,T_0)$ and  $\varepsilon>0$ there exists and integer $m_\ep\geqslant 1$ and a real number $\tau_\ep\in(0,T_0-\bar t)$ such that
\begin{equation}
\label{eq:noncritical}
\frac{\mathbf{x}_m(t)-\mathbf{x}_m(\bar t)}{t-\bar t}\in \mbox{\em  co }
H_p\big(\mathbf{x}(\bar t),Du^+(\mathbf{x}(\bar t))\big)+\ep B\quad \forall m\geqslant m_\ep\,,\;\forall t\in[\bar t,\bar t+\tau_\ep],
\end{equation}
where $B$ benotes the closed unit ball of $\R^2$, centered at the origin. 
\end{Lem}
\begin{proof}
We begin by showing that for every $\bar t\in [0,T_0)$ and  $\varepsilon>0$
 there exist $m_\ep\geqslant 1$ and  $\tau_\ep\in(0,T_0-\bar t)$ satisfying
\begin{equation}\label{eq:Young}
\dot{\mathbf{x}}_m(t)\in 
H_p\big(\mathbf{x}(\bar t),Du^+(\mathbf{x}(\bar t))\big)+\ep B,\qquad t\in[\bar t,\bar t+\tau_\ep]\;\mbox{  a.e.}
\end{equation}
 for all $ m\geqslant m_\ep$. We argue by contradiction:  set $\Phi(\bar t)= H_p\big(\mathbf{x}(\bar t),Du^+(\mathbf{x}(\bar t))\big)$ and suppose there exist $\bar t\in [0,T_0)$,  $\varepsilon>0$, and sequences $m_k\to\infty$ and $t_k\downarrow \bar t$ such that
 \begin{equation*}
\begin{cases}
(i) &  \dot{\mathbf{x}}_{m_k}(t_k)\notin \Phi(\bar t)+\ep B,\quad \forall k\geqslant 1
 \\
(ii)&Du_{m_k}\big( \mathbf{x}_{m_k}(t_k)\big)\to \bar p
\quad(k\to\infty)
\end{cases}
\end{equation*}
where we have used bound $(b)$ above to justify $(ii)$. We claim that $\bar p\in D^+u\big(\mathbf{x}(\bar t)\big)$.
Indeed, in view of (c) above we have that, for all $k\geqslant 1$,
\begin{equation*}
u_{m_k}\big( \mathbf{x}_{m_k}(t_k)+y\big)-u_{m_k}\big( \mathbf{x}_{m_k}(t_k)\big)-\big\langle Du_{m_k}\big( \mathbf{x}_{m_k}(t_k)\big),y\big\rangle 
\leqslant C_2|y|^2,\quad\forall |y|\leqslant R_0.
\end{equation*}
Hence, in the limit as $k\to\infty$, we get
\begin{equation*}
u( \mathbf{x}(\bar t)+y)-u( \mathbf{x}(\bar t))-\langle \bar p,y\rangle 
\leqslant C_2|y|^2,\quad\forall |y|\leqslant R_0,
\end{equation*}
which in turn proves our claim. Thus, we conclude that
\begin{equation*}
 \dot{\mathbf{x}}_{m_k}(t_k)=H_p\big( \mathbf{x}_{m_k}(t_k),Du_{m_k}( \mathbf{x}_{m_k}(t_k))\big)
 \stackrel{k\to\infty}{\longrightarrow} H_p(\mathbf{x}(\bar t),\bar p)\in\Phi(\bar t)
\end{equation*}
in contrast with $(i)$. So, \eqref{eq:Young} is proved.

Finally, \eqref{eq:noncritical} can be derived from \eqref{eq:Young} by integration.
\end{proof}

\smallskip
By appealing to the upper semi-continuity
 of $D^+u$ and assumption \eqref{hp:noncritical} we conclude that there exists $T\in(0,T_0]$ such that 
 \begin{equation}
\label{eq:noncritical_t}
0\notin  \mbox{  co }H_p\big(\mathbf{x}( t),Du^+(\mathbf{x}( t))\big)\quad\forall t\in[0,T].
\end{equation}
Now,  
 fix any $\bar t\in [0,T)$ and let $\bar{v}\in\R^2$ be any vector such that
\begin{equation}\label{eq:limits}
	\lim_{j\to\infty}\frac{\mathbf{x}(\bar t+\tau_j)-\mathbf{x}(\bar t)}{\tau_j}=\bar{v}
\end{equation}
 for some sequence $\tau_j\searrow 0$ ($j\to\infty$). Observe that  $\bar v\in \mbox{  co }H_p\big(\mathbf{x}(\bar t),Du^+(\mathbf{x}(\bar t))\big)$
  in view of Lemma~\ref{lma:noncritical}. So,
  $\bar v\neq 0$ owing to \eqref{eq:noncritical_t}.
 Set $\bar x=\mathbf{x}(\bar t)$ and define
 \begin{eqnarray*}
&\bar{p}\in\R^2&  \mbox{ by }\bar{v}=H_p(\bar x,\bar{p}) \mbox{ (or $\bar{p}=L_v(\bar{x},\bar{v})$)}
\\
&F_{\bar{v}}(\bar{x})=&\big\{p^*\in D^+u(\bar{x}):\langle p^*,\bar v\rangle =\min_{p\in D^+u(\bar{x})}\langle p,\bar v\rangle \big\}.
\end{eqnarray*}
Notice that $F_{\bar{v}}(\bar{x})$ is  the exposed face of the convex set $D^+u(\bar{x})$ in the direction $\bar{v}$ (see, for instance, \cite{Cannarsa_Sinestrari_book}). 
The following lemma identifies $\bar p$ (hence $\bar v$) uniquely.
\begin{Lem}\label{lem0_app}
Suppose  $\bar{p}\in F_{\bar{v}}(\bar{x})$. Then $\bar{p}$ is the unique element in $D^+u(\bar{x})$ such that
\begin{equation}\label{eq:min_e}
	H(\bar{x},\bar{p})=\min_{p\in D^+u(\bar{x})}H(\bar{x},p).
\end{equation}
\end{Lem}

\begin{proof}
Since $\bar{p}\in F_{\bar{v}}(\bar{x})$, we have that
\begin{align*}
	\langle  \bar p,\bar v\rangle =\langle \bar p,H_p(\bar{x},\bar{p})\rangle  =\min_{p\in D^+u(\bar{x})}\langle p, H_p(\bar{x},\bar{p})\rangle.
\end{align*}
Therefore, by convexity we conclude that
\begin{align*}
	0\leqslant \langle H_p(\bar{x},\bar{p}),p-\bar{p}\rangle\leqslant H(\bar{x},p)-H(\bar{x},\bar{p}),\quad\forall p\in D^+u(\bar{x}). 
\end{align*}
Since $H$ is strictly convex in $p$, $\bar{p}$ is the unique element in $D^+u(\bar{x})$ satisfying \eqref{eq:min_e}. 
\end{proof}

\noindent
Notice that the above lemma yields the existence of the right-derivative $\dot{\mathbf{x}}^+(\bar t)$ as soon as one shows that $\bar{p}\in F_{\bar{v}}(\bar{x})$ for any $\bar{v}$ satisfying \eqref{eq:limits}. 

\smallskip
Next, to show that $\bar{p}\in F_{\bar{v}}(\bar{x})$, we proceed by contradiction assuming that
\begin{equation}\label{eq:ast2}
	\bar{p}\not\in F_{\bar{v}}(\bar{x}).
\end{equation}
Let us define functions $\alpha,\beta:D^+u(\bar{x})\to\R$ by
\begin{align*}
	\alpha(p)=\langle p,\bar v\rangle -\frac{\partial u}{\partial\bar{v}}(\bar{x}),\quad \beta(x,p)=
	\langle p-\bar{p},H_p(x,p)-H_p(x,\bar{p})\rangle 
	\quad\forall p\in D^+u(\bar{x})
\end{align*}
where we have set $\frac{\partial u}{\partial\bar{v}}(\bar{x})=\lim_{\lambda\to0^+}\frac{u(\bar{x}+\lambda\bar{v})-u(\bar{x})}{\lambda}$. Recall that, since $u$ is semiconcave, 
\begin{equation}\label{eq:directional_derivative}
	\frac{\partial u}{\partial\bar{v}}(\bar{x})=\min_{p\in D^+u(\bar{x})}\langle p,\bar v\rangle 
\end{equation}
(see, for instance, \cite{Cannarsa_Sinestrari_book}). The following simple lemma is crucial for the proof.

\begin{Lem}\label{lem1_app}
If $\bar{p}\not\in F_{\bar{v}}(\bar{x})$, then
\begin{align*}
	\mu:=\min_{p\in D^+u(\bar{x})}\{\alpha(p)+\beta(\bar{x},p)\}>0.
\end{align*}
\end{Lem}

\begin{proof}
Observe first that $\beta(x,p)\geqslant0$ by convexity and $\alpha(p)\geqslant0$ for all $p\in D^+u(\bar{x})$ by \eqref{eq:directional_derivative}. Since we suppose $\bar{p}\not\in F_{\bar{v}}(\bar{x})$, just  two cases are possible.
\begin{enumerate}[(1)]
	\item If $\bar{p}\not\in D^+u(\bar{x})$, then $p\not=\bar{p}$ for all $p\in D^+u(\bar{x})$. So $\beta(\bar{x},p)>0$ by strict convexity.
	\item If $\bar{p}\in D^+u(\bar{x})\setminus F_{\bar{v}}(\bar{x})$, then $\alpha(p)>0$.
\end{enumerate}
In conclusion, \
\begin{align*}
	M(p):=\alpha(p)+\beta(\bar{x},p)>0,\quad \forall p\in D^+u(\bar{x}).
\end{align*}
Since $M$ is continuous and $D^+u(\bar{x})$ is compact, the conclusion follows.
\end{proof}

For any $\varepsilon>0$ set 
$$F^{\varepsilon}_{\bar{v}}(\bar{x})=F_{\bar{v}}(\bar{x})+\varepsilon B\quad\mbox{and}\quad V_{\varepsilon}=D^+u(\bar{x})+\varepsilon B.
$$
Now, let us fix $\varepsilon=\ep(\bar v,\mu)>0$  such that
\begin{equation}\label{eq:ast3}
	\bar{p}\not\in F^{\varepsilon}_{\bar{v}}(\bar{x})\quad \text{and} \quad \min_{p\in V_{\varepsilon}}\{\alpha(p)+\beta(\bar{x},p)\}\geqslant\frac 23\mu.
\end{equation}
Let $0<R\leqslant R_0$ be such that
\begin{align*}
	 D^+u(x)\subset V_{\varepsilon/2}\quad\forall x\in B(\bar x,R).
\end{align*}
Consider the line segment 
$$\gamma(t):=\bar{x}+(t-\bar t)\bar{v}\qquad (t\in[\bar t,T])$$ 
and fix $q\in(0,1)$. After possible reducing $T$, we can assume that 
\begin{align*}
	|\gamma(t)-\bar{x}|\leqslant qR\quad\text{and}\quad|\mathbf{x}(t)-\bar{x}|\leqslant qR\quad\forall t\in[\bar t,T].
\end{align*}
Consequently, there exists $\bar m\in\N$ such that for all $m\geqslant \bar m$ we have
\begin{enumerate}[(i)]
	\item $Du_m(x)\in V_{\varepsilon}$ for all $x\in B(\bar{x},R)$;
	\item $\mathbf{x}_m(t)\in B(\bar{x},R)$ for all $t\in[\bar t,T]$.
\end{enumerate}
Moreover, by  cutting $T$ down to size, we can have  the following property satisfied: 
\begin{enumerate}[(i)]\setliststart{3}
	\item for any $t\in[\bar t,T]$ there exists $m(t)\geqslant \bar m$ such that
	\begin{align}\label{eq:m(t)}
		d_{F_{\bar{v}}(\bar{x})}(Du_m(\gamma(t)))<\varepsilon,\quad\forall m\geqslant m(t).
	\end{align}
\end{enumerate}
We observe that (iii) is a consequence of Proposition 3.3.15 in \cite{Cannarsa_Sinestrari_book} since $\bar{v}\not=0$.

For $0<\delta$ to be chosen later on, we define
\begin{align*}
	K_{\delta}=&\,\bigcup_{\bar t\leqslant t\leqslant T}B(\gamma(t),\delta (t-\bar t))\\
	=&\,\{x\in\R^n: \text{there exists}\ t\in[\bar t,T]\ \text{such that}\ |x-\gamma(t)|\leqslant\delta (t-\bar t)\}.
\end{align*}

\begin{Lem}\label{lem2_app}
Let $\varepsilon>0$ and $m(\cdot)$ be fixed so that  \eqref{eq:ast3} and \eqref{eq:m(t)} hold true. If $\bar{p}\not\in F_{\bar{v}}(\bar{x})$, then there exists $\delta>0$ such that for all $j$ sufficiently large, $\mathbf{x}_m(t)\not\in K_{\delta}$ for all $t\in(\bar t+3\tau_j,T)$ and $m$ sufficiently large.
\end{Lem}

\begin{proof}
Throughout this proof  $j\in\N$ is supposed to be so large  that $\tau_j<(T-\bar t)/3$. Moreover, in order to simplify the notation, abbreviate  $\tau$ for $\tau_j$ and we assume $\bar t=0$. 

For all $t\in(3\tau,T)$ we have that
\begin{align*}
	&\,\frac{d}{dt}\big(u_m(\mathbf{x}_m(t))-\langle  \bar p,\mathbf{x}_m(t)
	\rangle \big)\\
	=&\,\big\langle Du_m(\mathbf{x}_m(t))-\bar{p},\dot{\mathbf{x}}_m(t) \big\rangle
	 =\big\langle H_p(\mathbf{x}_m(t),Du_m(\mathbf{x}_m(t))),Du_m(\mathbf{x}_m(t))-\bar{p}\big\rangle 
.
\end{align*}
Therefore, by integrating on $(\tau,t)$,
\begin{equation}\label{eq:1}
	\begin{split}
		&\,u_m(\mathbf{x}_m(t))-\langle \bar p,\mathbf{x}_m(t)\rangle 
			-u_m(\mathbf{x}_m(\tau))+\langle \bar p,\mathbf{x}_m(\tau)\rangle\\
		=&\,\int^t_{\tau}
		\big\langle H_p(\mathbf{x}_m(s),Du_m(\mathbf{x}_m(s))), Du_m(\mathbf{x}_m(s))-\bar{p}\big\rangle 
	ds.
	\end{split}
\end{equation}
Similarly, 
\begin{align*}
	\frac{d}{dt}\big(u_m(\gamma(t)))-
	\langle \bar p,\gamma(t)\rangle\big) =
	\langle Du_m(\gamma(t))-\bar{p},\bar v\rangle .
\end{align*}
So, (iii) and Lebesgue's theorem ensure that
\begin{equation}\label{eq:2}
	\begin{split}
		&\,u_m(\gamma(t))-
		\langle \bar p,\gamma(t)
		\rangle 
		-u_m(\gamma(\tau))+
		\langle \bar p,\gamma(\tau)
		\rangle 
		\\
		=&\,
		\int^t_{\tau}
		\langle Du_m(\gamma(s))-\bar{p},\bar v\rangle 
		\ ds
		\leqslant \bigg(\frac{\partial u}{\partial\bar{v}}(\bar{x})-
		\langle \bar p,\bar v\rangle 
		+\varepsilon|\bar{v}|\bigg)(t-\tau).
	\end{split}
\end{equation}
Therefore, by \eqref{eq:1} and \eqref{eq:2} we obtain 
\begin{align*}
	&\,u_m(\mathbf{x}_m(t))-\langle \bar p,\mathbf{x}_m(t)\rangle 
			-u_m(\mathbf{x}_m(\tau))+\langle \bar p,\mathbf{x}_m(\tau)\rangle\\
	&\,-\big(\,u_m(\gamma(t))-
		\langle \bar p,\gamma(t)
		\rangle 
		-u_m(\gamma(\tau))+
		\langle \bar p,\gamma(\tau)
		\rangle \big)\\
	\geqslant&\,\int^t_{\tau}\bigg\{\big\langle H_p(\mathbf{x}_m(s),Du_m(\mathbf{x}_m(s))), Du_m(\mathbf{x}_m(s))-\bar{p}\big\rangle -\bigg(\frac{\partial u}{\partial\bar{v}}(\bar{x})-
		\langle \bar p,\bar v\rangle 
		+\varepsilon|\bar{v}|\bigg)\bigg\}ds
\end{align*}
which can be rewritten as
\begin{equation}\label{eq:app1}
	\begin{split}
		&\,\,u_m(\mathbf{x}_m(t))-\langle \bar p,\mathbf{x}_m(t)\rangle 
			-u_m(\mathbf{x}_m(\tau))+\langle \bar p,\mathbf{x}_m(\tau)\rangle\\
	&\,-\big(\,u_m(\gamma(t))-
		\langle \bar p,\gamma(t)
		\rangle 
		-u_m(\gamma(\tau))+
		\langle \bar p,\gamma(\tau)
		\rangle \big)\\
	\geqslant&\,
	\int^t_{\tau}
	\big\langle H_p(\mathbf{x}_m(s),Du_m(\mathbf{x}_m(s)))-\bar v, Du_m(\mathbf{x}_m(s))-\bar{p}\big\rangle
		ds
		\\
	&\,+\int^t_{\tau}\bigg(\langle Du_m(\mathbf{x}_m(s)),\bar{v}\rangle-\frac{\partial u}{\partial\bar{v}}(\bar{x})-\varepsilon|\bar{v}|\bigg)\ ds 
	\\
	=&\,
	\int^t_{\tau}
	\big\langle H_p(\mathbf{x}_m(s),Du_m(\mathbf{x}_m(s)))-H_p(\mathbf{x}_m(s), \bar{p}),Du_m(\mathbf{x}_m(s))-\bar{p}\big\rangle
		ds
		\\
	&\,+\int^t_{\tau}\big\langle H_p(\mathbf{x}_m(s),\bar{p})-H_p(\bar{x},\bar{p}),Du_m(\mathbf{x}_m(s))-\bar{p}\big\rangle \ ds
	\\
	&\,+\int^t_{\tau}\bigg(\langle Du_m(\mathbf{x}_m(s)),\bar{v}\rangle-\frac{\partial u}{\partial\bar{v}}(\bar{x})-\varepsilon|\bar{v}|\bigg)\ ds 
	\\
	\geqslant&\,\int^t_{\tau}\big\{\alpha(Du_m(\mathbf{x}_m(s)))+\beta(\mathbf{x}_m(s),Du_m(\mathbf{x}_m(s)))-\varepsilon|\bar{v}|\big\} ds
	\\
	&\,+\int^t_{\tau}\big\langle H_p(\mathbf{x}_m(s),\bar{p})-H_p(\bar{x},\bar{p}),Du_m(\mathbf{x}_m(s))-\bar{p}\big\rangle \ ds
	.
	\end{split}
\end{equation}
Now, observe the following:
\begin{align*}
|u_m(\gamma(t))-
		\langle \bar p,\gamma(t)
		\rangle 
		-u_m(\gamma(\tau))+
		\langle \bar p,\gamma(\tau)
		\rangle|
	\leqslant\,(C_1+|\bar{p}|)|\gamma(\tau)-\mathbf{x}_m(\tau)|\\
	\leqslant\,(C_1+|\bar{p}|)(|\gamma(\tau)-\bar{\gamma}(\tau)|+|\bar{\gamma}(\tau)-\mathbf{x}_m(\tau)|)
\end{align*}
where we recall that $C_1\geqslant\|Du_m\|_{\infty}$.

Next, we fix $\tau=\tau_j$ with $j$ large enough so that
\begin{align*}
	|\gamma(\tau)-\bar{\gamma}(\tau)|\leqslant\frac{\delta\tau}{2}\quad\text{and}\quad m\gg1\ \text{so that}\ |\bar{\gamma}(\tau)-\gamma_m(\tau)|\leqslant\frac{\delta\tau}{2}.
\end{align*}
Then
\begin{equation}\label{eq:app2}
	|u_m(\gamma(\tau))-
		\langle \bar p,\gamma(\tau)
		\rangle 
		-u_m(\mathbf{x}_m(\tau))+
		\langle \bar p,\mathbf{x}_m(\tau)
		\rangle|
	\leqslant\,(C_1+|\bar{p}|)\delta\tau.
\end{equation}
%
%
Since $Du_m(\mathbf{x}_m(s))\in V_{\varepsilon}$ for all $s\in[0,\tau]$, by \eqref{eq:ast3} we have that
\begin{equation}\label{eq:app3}
	\int^t_{\tau}\bigg\{\alpha(Du_m(\mathbf{x}_m(s)))+\beta(\mathbf{x}_m(s),Du_m(\mathbf{x}_m(s)))\bigg\} ds\geqslant\frac 23\mu(t-\tau).
\end{equation}
We also have that, after cutting down on $T>0$,
\begin{equation}\label{eq:app41}
	\begin{split}
	&\,
	\langle 
	H_p(\mathbf{x}_m(s),\bar{p})-H_p(\bar{x},\bar{p}),Du_m(\mathbf{x}_m(s))-\bar{p}
	\rangle \\
	\geqslant&\,-(C_1+|\bar{p}|)\cdot C_2|\mathbf{x}_m(s)-\bar{x}|\\
	\geqslant&\,-\varepsilon C_2(C_1+|\bar{p}|)
	\end{split}
\end{equation}
So, by \eqref{eq:app1}, \eqref{eq:app2}, \eqref{eq:app3} and \eqref{eq:app41} we conclude that
\begin{align*}
	&\,u_m(\mathbf{x}_m(t))-
	\langle \bar p,\mathbf{x}_m(t)\rangle 
		-u_m(\gamma(t))+
	\langle \bar p,\gamma(t)\rangle 	
			\\
	\geqslant&\,\bigg(\frac 23\mu-\varepsilon(|\bar{v}|+C_2(C_1+|\bar{p}|))\bigg)(t-\tau)-\delta\tau(C_1+|\bar{p}|).
\end{align*}
On the other hand,
\begin{align*}
	|u_m(\mathbf{x}_m(t))-
	\langle \bar p,\mathbf{x}_m(t)\rangle 
		-u_m(\gamma(t))+
	\langle \bar p,\gamma(t)\rangle |\leqslant(C_1+|\bar{p}|)|\mathbf{x}_m(t)-\gamma(t)|.
\end{align*}
Therefore,
\begin{equation}
	|\mathbf{x}_m(t)-\gamma(t)|\geqslant\frac{2\mu/3-\varepsilon(|\bar{v}|+C_2(C_1+|\bar{p}|))}{C_1+|\bar{p}|}(t-\tau)-\delta\tau.
\end{equation}
We now take $0\leqslant\varepsilon(|\bar{v}|+C_2(C_1+|\bar{p}|))|<\frac{\mu}3$ to obtain
\begin{align*}
	|\mathbf{x}_m(t)-\gamma(t)|\geqslant\frac{\mu(t-\tau)}{3(C_1+|\bar{p}|)}-\delta\tau
\end{align*}
and look for $t<T$ such that
\begin{equation}\label{eq:app4}
	\frac{\mu(t-\tau)}{3(C_1+|\bar{p}|)}-\delta\tau\geqslant2\delta t.
\end{equation}
So, taking $0<\delta\leqslant\frac{\mu}{12(C_1+|\bar{p}|)}$, we have that \eqref{eq:app4} is satisfied for all 
$$t\geqslant\frac{2(\delta+\mu)}{\mu}\tau=2(1+\frac{\delta}{\mu})\tau.
$$
Finally, $\frac{\delta}{\mu}\leqslant\frac{1}{12(C_1+|\bar{p}|)}$ with $C_1\geqslant 1$ gives that \eqref{eq:app4} holds for all $t\in[3\tau,T]$.
\end{proof}
To complete the proof it suffices to note that Lemma \ref{lem1_app} and Lemma \ref{lem2_app} ensure that assuming \eqref{eq:ast2} leads to a contradiction. Indeed,
\begin{align*}
	|\mathbf{x}_m(t)-\gamma(t)|\geqslant2\delta t,\quad\forall t\in[3\tau_j,T], \forall j\gg1
\end{align*}
implies that $\mathbf{x}(t)\not\in K_{\delta}$ for all $t\in[3\tau_j,T]$. On the other hand, $\mathbf{x}(\tau_i)\in K_{\delta}$ for $i\gg1$ and, for any fixed $i$, $\tau_i\in[3\tau_j,T]$ for $j$ sufficiently large. 
\end{proof}

\begin{center}
\sc Statement
\end{center}
On behalf of all authors, the corresponding author states that there is no conflict of interest.

\bibliographystyle{abbrv}
\bibliography{mybib}

\begin{thebibliography}{10}

\bibitem{Albano_Cannarsa2002}
P.~Albano and P.~Cannarsa.
\newblock Propagation of singularities for solutions of nonlinear first order
  partial differential equations.
\newblock {\em Arch. Ration. Mech. Anal.}, 162(1):1--23, 2002.

\bibitem{ACNS2013}
P.~Albano, P.~Cannarsa, K.~T. Nguyen, and C.~Sinestrari.
\newblock Singular gradient flow of the distance function and homotopy
  equivalence.
\newblock {\em Math. Ann.}, 356(1):23--43, 2013.

\bibitem{Cannarsa_Chen_Cheng2019}
P.~Cannarsa, Q.~Chen, and W.~Cheng.
\newblock Dynamic and asymptotic behavior of singularities of certain weak
  {KAM} solutions on the torus.
\newblock {\em J. Differential Equations}, 267(4):2448--2470, 2019.

\bibitem{Cannarsa_Cheng3}
P.~Cannarsa and W.~Cheng.
\newblock Generalized characteristics and {L}ax-{O}leinik operators: global
  theory.
\newblock {\em Calc. Var. Partial Differential Equations}, 56(5):Art. 125, 31,
  2017.

\bibitem{Cannarsa_Cheng_Fathi2017}
P.~Cannarsa, W.~Cheng, and A.~Fathi.
\newblock On the topology of the set of singularities of a solution to the
  {H}amilton-{J}acobi equation.
\newblock {\em C. R. Math. Acad. Sci. Paris}, 355(2):176--180, 2017.

\bibitem{Cannarsa_Cheng_Fathi2019}
P.~Cannarsa, W.~Cheng, and A.~Fathi.
\newblock Singularities of solutions of time dependent {H}amilton-{J}acobi
  equations. {A}pplications to {R}iemannian geometry.
\newblock preprint, arXiv:1912.04863, 2019.

\bibitem{CCMW2019}
P.~Cannarsa, W.~Cheng, M.~Mazzola, and K.~Wang.
\newblock Global generalized characteristics for the {D}irichlet problem for
  {H}amilton-{J}acobi equations at a supercritical energy level.
\newblock {\em SIAM J. Math. Anal.}, 51(5):4213--4244, 2019.

\bibitem{Cannarsa_Cheng_Zhang2014}
P.~Cannarsa, W.~Cheng, and Q.~Zhang.
\newblock Propagation of singularities for weak {KAM} solutions and barrier
  functions.
\newblock {\em Comm. Math. Phys.}, 331(1):1--20, 2014.

\bibitem{Cannarsa_Mazzola_Sinestrari2015}
P.~Cannarsa, M.~Mazzola, and C.~Sinestrari.
\newblock Global propagation of singularities for time dependent
  {H}amilton-{J}acobi equations.
\newblock {\em Discrete Contin. Dyn. Syst.}, 35(9):4225--4239, 2015.

\bibitem{Cannarsa_Sinestrari_book}
P.~Cannarsa and C.~Sinestrari.
\newblock {\em Semiconcave functions, {H}amilton-{J}acobi equations, and
  optimal control}, volume~58 of {\em Progress in Nonlinear Differential
  Equations and their Applications}.
\newblock Birkh{\"a}user Boston, Inc., Boston, MA, 2004.

\bibitem{Cannarsa_Yu2009}
P.~Cannarsa and Y.~Yu.
\newblock Singular dynamics for semiconcave functions.
\newblock {\em J. Eur. Math. Soc. (JEMS)}, 11(5):999--1024, 2009.

\bibitem{Fathi_Maderna2007}
A.~Fathi and E.~Maderna.
\newblock Weak {KAM} theorem on non compact manifolds.
\newblock {\em NoDEA Nonlinear Differential Equations Appl.}, 14(1-2):1--27,
  2007.

\bibitem{Khanin_Sobolevski2016}
K.~Khanin and A.~Sobolevski.
\newblock On dynamics of {L}agrangian trajectories for {H}amilton-{J}acobi
  equations.
\newblock {\em Arch. Ration. Mech. Anal.}, 219(2):861--885, 2016.

\bibitem{Rifford2008}
L.~Rifford.
\newblock On viscosity solutions of certain {H}amilton-{J}acobi equations:
  regularity results and generalized {S}ard's theorems.
\newblock {\em Comm. Partial Differential Equations}, 33(1-3):517--559, 2008.

\bibitem{Stromberg2011}
T.~Str{\"o}mberg.
\newblock A counterexample to uniqueness of generalized characteristics in
  {H}amilton-{J}acobi theory.
\newblock {\em Nonlinear Anal.}, 74(7):2758--2762, 2011.

\bibitem{Stromberg2013}
T.~Str{\"o}mberg.
\newblock Propagation of singularities along broken characteristics.
\newblock {\em Nonlinear Anal.}, 85:93--109, 2013.

\bibitem{Stromberg_Ahmadzadeh2014}
T.~Str{\"o}mberg and F.~Ahmadzadeh.
\newblock Excess action and broken characteristics for {H}amilton-{J}acobi
  equations.
\newblock {\em Nonlinear Anal.}, 110:113--129, 2014.

\bibitem{Yu2006}
Y.~Yu.
\newblock A simple proof of the propagation of singularities for solutions of
  {H}amilton-{J}acobi equations.
\newblock {\em Ann. Sc. Norm. Super. Pisa Cl. Sci. (5)}, 5(4):439--444, 2006.

\end{thebibliography}

\end{document}